\documentclass[10pt]{article}
\usepackage{amssymb}
\usepackage{amsmath}
\usepackage{amsthm}
\usepackage{amsfonts}

\usepackage{a4wide}
\usepackage{longtable}
\usepackage{multirow}

\usepackage{graphicx}
\usepackage{stmaryrd}

\usepackage[latin1]{inputenc}
\usepackage[english]{babel}
\usepackage[T1]{fontenc}
\usepackage{amssymb,amsmath,amsthm,amstext,amsfonts,latexsym}

\usepackage{pgf,tikz}
\usepackage{mathrsfs}
\usetikzlibrary{arrows}

\newtheorem{lemma}{Lemma}[section]

\newtheorem{theorem}{Theorem}[section]
\newtheorem{corollary}{Corollary}[theorem]
\newtheorem{conjecture}{Conjecture}[section]

\theoremstyle{definition}
\newtheorem{definition}{Definition}[section]

\theoremstyle{remark}
\newtheorem{remark}{Remark}[section]
\newtheorem{example}{Example}[section]

\usepackage[all]{xy}
\input xy
\xyoption {all}
\usepackage[pdftex,             %
    bookmarks=true,             
    hypertexnames=false,        
    bookmarksnumbered=true,     
    pdfpagemode=None,           
    pdfstartview=FitH,          
    pdfpagelayout=SinglePage,   
    colorlinks=true,            
    urlcolor=blue,              
    citecolor=blue,             
    pdfborder={0 0 0}           
    ]{hyperref}

\date{ }

\begin{document}

\title{Principal factors and lattice minima}

\author{S. AOUISSI, A. AZIZI, M. C. ISMAILI, D. C. MAYER and M. TALBI}


\maketitle

\noindent
\begin{center}
Dedicated to H. C. Williams.
\end{center}
\bigskip


\medskip
\noindent
\textbf{Abstract:}
Let $\mathit{k}=\mathbb{Q}(\sqrt[3]{d},\zeta_3)$, where $d>1$ is a cube-free positive integer,
$\mathit{k}_0=\mathbb{Q}(\zeta_3)$ be the cyclotomic field containing a primitive cube root of unity $\zeta_3$, 
and $G=\operatorname{Gal}(\mathit{k}/\mathit{k}_0)$. The possible prime factorizations of $d$ in our main result  \cite[Thm. 1.1]{AMI} give rise to new phenomena
concerning the chain $\Theta=(\theta_i)_{i\in\mathbb{Z}}$ of \textit{lattice minima} in the underlying pure cubic subfield $L=\mathbb{Q}(\sqrt[3]{d})$ of $\mathit{k}$.
The aims of the present work are
 to give criteria for the occurrence of generators of primitive ambiguous principal ideals
 $(\alpha)\in\mathcal{P}_{\mathit{k}}^G/\mathcal{P}_{\mathit{k}_0}$
 among the lattice minima $\Theta=(\theta_i)_{i\in\mathbb{Z}}$ of the underlying pure cubic field $L=\mathbb{Q}(\sqrt[3]{d})$,
 and to explain exceptional behavior of the chain $\Theta$ for certain radicands $d$
 with impact on determining the principal factorization type of $L$ and $\mathit{k}$ by means of Voronoi's algorithm.
 

\bigskip


\section{Introduction}
\label{s:Intro}
Let $\mathit{k}=\mathbb{Q}(\sqrt[3]{d},\zeta_3)$, where $d>1$ is a cube-free positive integer,
$\mathit{k}_0=\mathbb{Q}(\zeta_3)$, where $\zeta_3$ is a primitive cube root of unity,
and $\mathit{k}^{*}$ be the relative genus field of $\mathit{k}/\mathit{k}_0$.
  \paragraph{}
   In our previous work \cite{AMI}, we implemented Gerth's methods \cite{Ge1976} and \cite{Ge1975}
for determining the rank of the group of ambiguous ideal classes of $\mathit{k}/\mathit{k}_0$ and
obtained all integers $d$ and conductors $f$ for which $\operatorname{Gal}(\mathit{k}^{*}/\mathit{k})\cong\mathbb{Z}/3\mathbb{Z}\times\mathbb{Z}/3\mathbb{Z}$.
In contrast with the radicands $d$ of the shape in \cite[Thm. 1.1]{AMITA},
the possible prime factorizations of $d$ in our main result  \cite[Thm. 1.1]{AMI}
are more complicated and give rise to new phenomena
concerning the chain $\Theta=(\theta_i)_{i\in\mathbb{Z}}$ of \textit{lattice minima}
\cite{HCW1980}, \cite{HCW1981}, \cite{HCW1982}
in the underlying pure cubic subfield $L=\mathbb{Q}(\sqrt[3]{d})$ of $\mathit{k}$.
A lattice minimum $\theta_i$ is an algebraic integer
with norm not exceeding the Minkowski bound of the maximal order $\mathcal{O}_{L}$ of $L$.
In particular, all positive units $\eta>0$ of $L$, which have norm $1$, are lattice minima
and the original purpose of Voronoi's algorithm \cite{GFV} was to find the \textit{fundamental unit} $0<\varepsilon<1$
by constructing the chain $\Theta$ and stopping at the first unit encountered, which must be $\varepsilon$.
\paragraph{}
More recently, however, it was the idea of Barrucand, Cohn \cite{BC1971} and Williams \cite{HCW1982}
to use Voronoi's algorithm for the classification of pure cubic fields into \textit{three principal factorization types},
which we have rederived with cohomological techniques in \cite[\S\ 2.1]{AMITA}.
The clue was to keep track of the norms $n_i=N_{L/\mathbb{Q}}(\theta_i)$ of all lattice minima
on the way through the chain $\Theta$, starting at the trivial unit $\theta_0=1$.
When some $n_i$ divides the square of the \textit{conductor} $f$ of $\mathit{k}/\mathit{k}_0$,
then $\theta_i$ is generator of a primitive ambiguous principal ideal in
$\mathcal{P}_{L}^G/\mathcal{P}_{\mathbb{Q}}\le\mathcal{P}_{\mathit{k}}^G/\mathcal{P}_{\mathit{k}_0}$,
more precisely an \textit{absolute principal factor},
and $L$ is of type $\beta$ \cite[Thm. 2.1]{AMITA}.
Now, the new phenomenon which arises for numerous radicands of the form in Equation (1) of \cite[Thm. 1.1]{AMI}
is the occasional failure of the chain $\Theta$ to lead to an absolute principal factor
although $L$ is of type $\beta$.
\paragraph{}
After explaining the connection between radicand $d$, conductor $f$, and ramification in $\mathit{k}/\mathit{k}_0$
in section \ref{ss:Ramification}, the formalism of canonical divisors in section \ref{ss:Canonical},
and the concept of lattice minima in section \ref{ss:LatticeMinima},
we prove necessary and sufficient conditions
for the occurrence of generators $\alpha$ of primitive ambiguous principal ideals
$(\alpha)\in\mathcal{P}_{L}^G/\mathcal{P}_{\mathbb{Q}}$
among the lattice minima in the chain $\Theta=(\theta_i)_{i\in\mathbb{Z}}$
in section \ref{ss:PrincipalFactorMinima}.
We develop a powerful new algorithm which elegantly avoids all mentioned problems
by using a non-maximal order $\mathcal{O}_{L,0}$ for $d\equiv\pm 1\,(\mathrm{mod}\,9)$,
and by exploiting the impossibility of type $\gamma$ \cite[Thm. 2.1]{AMITA} for $d\equiv\pm 2,\pm 4\,(\mathrm{mod}\,9)$,
in section \ref{ss:Algorithm}, and we give an explicit criteria for M0-fields in rational integers in section \ref{ss:M0Fields}.
\paragraph{}
The new techniques were implemented for an extensive classification of all normalized radicands $2\le d<10^6$
and they detected serious defects in the previous table \cite[\S\ 6, p. 272, and Tbl. 2, p. 273]{HCW1982}. The usual notations  is given as follows:
\begin{itemize}
 \item $L= \mathbb{Q}(\sqrt[3]{d})$ is a pure cubic field, where $d>1$ is a cube-free positive integer;
 \item $\mathit{k}_0= \mathbb{Q}(\zeta_3)$, where $\zeta_3=e^{2i\pi/3}$ denotes a primitive third root of unity;
 \item $\mathit{k}=\mathbb{Q}(\sqrt[3]{d},\zeta_3)$ is the normal closure of $L$;
   \item $f$ is the conductor of the relative Kummer extension $\mathit{k}/\mathit{k}_0$; 
   \item $v_{l}(x)$ is the $l$-valuation of the integer $x$;
   \item $\Theta=(\theta_i)_{i\in\mathbb{Z}}$ is the chain of lattice minima in the underlying pure cubic subfield $L$ of $\mathit{k}$.
   \item $Q$ is the index of the subgroup $E_0$ generated by the units of intermediate fields of the extension $\mathrm{k}/\mathbb{Q}$ in the group of units of $\mathrm{k}$;
    \item $\langle\tau\rangle=\operatorname{Gal}(\mathit{k}/L)$ such that $\tau^2=id$, $\tau(\zeta_3)=\zeta_3^2$ and $\tau(\sqrt[3]{d})=\sqrt[3]{d}$;
 \item $\langle\sigma\rangle=\operatorname{Gal}(\mathit{k}/\mathit{k}_0)$ such that $\sigma^3=id$, $\sigma(\zeta_3)=\zeta_3$,
   $\sigma(\sqrt[3]{d})=\zeta_3\sqrt[3]{d}$ and $\tau\sigma=\sigma^2\tau$;
 \item For an algebraic number field $F$:
  \begin{itemize}
  \item $\mathcal{O}_{F}$, $E_{F}$ : the ring of integers and the group of units of $F$;
  \item $\mathcal{I}_{F}$, $\mathcal{P}_{F}$ : the group of ideals and the subgroup of principal ideals of $F$;
  \end{itemize}
\end{itemize}






\section{Conductor and ramification}
\label{ss:Ramification}

Let $L=\mathbb{Q}(\sqrt[3]{d})$ be a pure cubic field
with \textit{normalized radicand} $d=a\cdot b^2$,
where $a>b\ge 1$ are square-free coprime integers.
The normalization enforces that the \textit{co-radicand} $\bar{d}=a^2\cdot b$
is strictly bigger than $d$.
It generates an isomorphic field $\mathbb{Q}(\sqrt[3]{\bar{d}})\simeq L$,
since $a^2\cdot b$ differs from the square $a^2\cdot b^4$ of $d$ by the complete third power $b^3$.

The class field theoretic \textit{conductor} $f$
of the associated relative Kummer extension $\mathit{k}/\mathit{k}_0$ is 
\begin{equation}
\label{eqn:Conductor}
f =
\begin{cases}
3ab & \text{ if } d\not\equiv\pm 1\,(\mathrm{mod}\,9) \text{ (Dedekind's species $1$)}, \\
 ab & \text{ if } d\equiv\pm 1\,(\mathrm{mod}\,9) \text{ (Dedekind's species $2$)}. \\
\end{cases}
\end{equation}
This means that all prime divisors of $ab$ are ramified in $\mathit{k}/\mathit{k}_0$.
If $L$ is of Dedekind's second species with $d\equiv\pm 1\,(\mathrm{mod}\,9)$,
then $3\nmid ab$ and $3$ is unramified in $\mathit{k}/\mathit{k}_0$.
However, if $L$ is of Dedekind's first species with $d\not\equiv\pm 1\,(\mathrm{mod}\,9)$,
then either $3\mid ab$ (species $1\mathrm{a}$) or $d\equiv\pm 2,\pm4\,(\mathrm{mod}\,9)$ (species $1\mathrm{b}$),
and in both cases $3$ is ramified in $\mathit{k}/\mathit{k}_0$
\cite[\S\ 2.2]{AMITA}.

For a prime number $\ell\in\mathbb{P}$,
we denote by $v_{\ell}:\,\mathbb{Q}\setminus\lbrace 0\rbrace\to\mathbb{Z}$
the $\ell$-valuation of non-zero rational numbers.

The species of the field $L$ can be expressed by the $3$-valuation of the conductor $f$:
\begin{equation}
\label{eqn:3Valuation}
v_3(f)=
\begin{cases}
2, \\
1, \\
0,
\end{cases}
\text{ if } L \text{ is of species }
\begin{cases}
1\mathrm{a}, \\
1\mathrm{b}, \\
2.
\end{cases}
\end{equation}
Since the conductor is divisible by $9$ for fields of species $1\mathrm{a}$,
it is convenient to define a \textit{ramification invariant} $R$
which is the product of all primes which are ramified in $\mathit{k}/\mathit{k}_0$:
\begin{equation}
\label{eqn:Ramification}
R:=
\begin{cases}
f = ab  & \text{ if } d\equiv\pm 1\,(\mathrm{mod}\,9) \text{ (and thus } 3\nmid ab), \\
f = 3ab & \text{ if } d\equiv\pm 2,\pm 4\,(\mathrm{mod}\,9) \text{ (and thus } 3\nmid ab), \\
f/3 = ab & \text{ if } 3\mid ab.
\end{cases}
\end{equation}


\section{Formalism of canonical divisors}
\label{ss:Canonical}

For the investigation of \textit{principal factors},
that is, generators $\alpha\in\mathcal{O}_L$ of primitive ambiguous principal ideals
$(\alpha)=\alpha\mathcal{O}_L\in\mathcal{P}_L^G/\mathcal{P}_{\mathbb{Q}}$,
which have divisors of the square $R^2$ of the ramification invariant $R$ as norms,
$n=\lvert N_{L/\mathbb{Q}}(\alpha)\rvert$ with $n\mid R^2$,
it is useful to introduce the formalism of \textit{canonical divisors}
of the radicand $d=ab^2$ with respect to the norm $n$
\cite[\S\ 7, p. 18]{BC1970}:
\begin{equation}
\label{eqn:Canonical}
\begin{aligned}
d_1 &:= \prod\lbrace\ell\in\mathbb{P}\mid v_{\ell}(a)=1,\ v_{\ell}(n)=1\rbrace, \quad
d_2 := \prod\lbrace\ell\in\mathbb{P}\mid v_{\ell}(a)=1,\ v_{\ell}(n)=2\rbrace, \\
d_4 &:= \prod\lbrace\ell\in\mathbb{P}\mid v_{\ell}(b)=1,\ v_{\ell}(n)=1\rbrace, \quad
d_5 := \prod\lbrace\ell\in\mathbb{P}\mid v_{\ell}(b)=1,\ v_{\ell}(n)=2\rbrace,
\end{aligned}
\end{equation}
and two additional \textit{silent} divisors for expressing the radicand and its components, 
\begin{equation}
\label{eqn:Silent}
\begin{aligned}
d_3 &:= \prod\lbrace\ell\in\mathbb{P}\mid v_{\ell}(a)=1,\ v_{\ell}(n)=0\rbrace, \quad
d_6 := \prod\lbrace\ell\in\mathbb{P}\mid v_{\ell}(b)=1,\ v_{\ell}(n)=0\rbrace.
\end{aligned}
\end{equation}
Then the norm $n$, the radicands $d,\bar{d}$, and their components $a,b$ have the following shape:
\begin{equation}
\label{eqn:Expressions}
n=d_1d_2^2d_4d_5^2,\ a=d_1d_2d_3,\ b=d_4d_5d_6,\ d=d_1d_2d_3d_4^2d_5^2d_6^2,\ \bar{d}=d_1^2d_2^2d_3^2d_4d_5d_6.
\end{equation}


\section{Lattice minima with principal factor norm}
\label{ss:LatticeMinima}

\noindent
We assume that $\sqrt[3]{d}$ denotes the unique \textit{real} zero of the pure equation $X^3-d=0$
and therefore the pure cubic field $L=\mathbb{Q}(\sqrt[3]{d})$ is a \textit{real} field with two complex conjugates,
$L^{\prime}=L^{\sigma}=\mathbb{Q}(\zeta_3\sqrt[3]{d})$ and $L^{\prime\prime}=L^{\sigma^2}=\mathbb{Q}(\zeta_3^2\sqrt[3]{d})$,
that is, with signature $(1,1)$ and torsion-free Dirichlet unit rank $1+1-1=1$.
Thus the total order of the field $\mathbb{R}$ of real numbers restricts to $L$,
which we shall need for investigating lattice minima.
We point out that the second \textit{algebraic} conjugate $\alpha^{\prime\prime}\in L^{\prime\prime}$ of an element $\alpha\in L$
is exactly the \textit{complex} conjugate of the first (algebraic) conjugate $\alpha^{\prime}\in L^{\prime}$ of $\alpha$,
since $L=\mathrm{Fix}(\tau)$, $\tau\sigma=\sigma^2\tau$, and thus
$\alpha^{\prime\prime}=\alpha^{\sigma^2}=(\alpha^{\tau})^{\sigma^2}=\alpha^{\sigma^2\tau}=\alpha^{\tau\sigma}=(\alpha^{\sigma})^{\tau}=(\alpha^{\prime})^{\tau}$
where $\tau$ with $\tau(\zeta_3)=\zeta_3^2=\bar{\zeta_3}$ is the complex conjugation restricted to $\mathit{k}$.

The Minkowski mapping 
$\psi:\,\mathcal{O}_{L}\to\mathbb{R}^3$, $\alpha\mapsto (\mathrm{Re}(\alpha^{\prime}),\mathrm{Im}(\alpha^{\prime}),\alpha)$
is an injective embedding of the maximal order $\mathcal{O}_{L}$ into Euclidean $3$-space $\mathbb{R}^3$.
The number geometric image $\psi(\mathcal{O}_{L})$ is a discrete free $\mathbb{Z}$-module of rank three, i.e., a complete lattice in $\mathbb{R}^3$.


\begin{definition}
\label{dfn:Minima}
The \textit{norm cylinder} of a point $\mathbf{x}=(x,y,z)$ in Euclidean $3$-space is defined by
\begin{equation}
\label{eqn:NormCylinder}
\mathcal{N}(\mathbf{x}):=\lbrace\mathbf{u}=(u,v,w)\in\mathbb{R}^3\mid 0\le u^2+v^2<x^2+y^2,\ 0\le w<\lvert z\rvert\rbrace.
\end{equation}
If $\mathcal{O}\subseteq\mathcal{O}_{L}$ is an order of the field $L$, not necessarily the maximal order,
then an algebraic integer $\alpha\in\mathcal{O}$ with $\alpha>0$ is called a \textit{lattice minimum} of $\mathcal{O}$ if 
\begin{equation}
\label{eqn:Minimum}
\mathcal{N}(\psi(\alpha))\bigcap\psi(\mathcal{O})=\lbrace\mathbf{O}\rbrace,
\text{ where } \mathbf{O}=(0,0,0) \text{ denotes the origin of } \mathbb{R}^3,
\end{equation}
or, equivalently, observing that
$\mathrm{Re}(\alpha^{\prime})^2+\mathrm{Im}(\alpha^{\prime})^2=\lvert\alpha^{\prime}\rvert^2=\alpha^{\prime}(\alpha^{\prime})^{\tau}=\alpha^{\prime}\alpha^{\prime\prime}$,
if
\begin{equation}
\label{eqn:MinimumVar}
(\forall\,\beta\in\mathcal{O})\,\Bigl(0\le\beta^{\prime}\beta^{\prime\prime}<\alpha^{\prime}\alpha^{\prime\prime},\ 0\le\beta<\alpha\Longrightarrow\beta=0\Bigr)
\end{equation}
\end{definition}

\noindent
Note that the volume of the cylinder is given by
$\mathrm{vol}_3(\mathcal{N}(\psi(\alpha)))=\pi\cdot\alpha^{\prime}\alpha^{\prime\prime}\cdot\alpha=\pi\cdot N_{L/\mathbb{Q}}(\alpha)$,
which justifies the designation \textit{norm cylinder}.

The set of all lattice minima of $\mathcal{O}$ is denoted by $\mathrm{Min}(\mathcal{O})$.


\begin{lemma}
\label{lem:Units}
All positive units in $E_{L}^+:=\lbrace\eta\in E_{L}\mid\eta>0\rbrace$ are lattice minima of $\mathcal{O}_{L}$,
but the radical $\delta:=\sqrt[3]{d}$ and the co-radical $\bar{\delta}:=\sqrt[3]{\bar{d}}$
never belong to $\mathrm{Min}(\mathcal{O}_{L})$.
More generally, if $\alpha\in\mathrm{Min}(\mathcal{O}_{L})$
then $\alpha\delta,\alpha\bar{\delta}\not\in\mathrm{Min}(\mathcal{O}_{L})$.
\end{lemma}

\begin{proof}
Let $\eta>0$ be a positive unit in $E_{L}=\langle -1,\varepsilon\rangle$,
where $\varepsilon>1$ denotes the fundamental unit of $L$.
For an algebraic integer $\alpha\in\mathcal{O}_{L}$ with $\psi(\alpha)\in\mathcal{N}(\psi(\eta))$,
we have $0\le\alpha^{\prime}\alpha^{\prime\prime}<\eta^{\prime}\eta^{\prime\prime}$ and $0\le\alpha<\eta$
and thus $0\le N_{L/\mathbb{Q}}(\alpha)<N_{L/\mathbb{Q}}(\eta)=1$.
Since $N_{L/\mathbb{Q}}(\alpha)\in\mathbb{Z}$ is an integer, this is only possible for $\alpha=0$.
Thus we have $\eta\in\mathrm{Min}(\mathcal{O}_{L})$.
In particular, the fundamental unit $\varepsilon$ and the trivial unit $1$ with $\psi(1)=(1,0,1)$
are lattice minima of $\mathcal{O}_{L}$.

Concerning the second claim,
which is also valid for any algebraic integer $\alpha\in\mathcal{O}_{L}$ with $\alpha>0$
(not necessarily $\alpha\in\mathrm{Min}(\mathcal{O}_{L})$),
we firstly observe that
$\delta,\bar{\delta}\ge\sqrt[3]{2}\approx 1.26>1$ since $d,\bar{d}\ge 2$,
furthermore
$N_{L/\mathbb{Q}}(\delta)=\delta\delta^{\prime}\delta^{\prime\prime}=\delta\cdot\zeta_3\delta\cdot\zeta_3^2\delta=\zeta_3^3\cdot\delta^3=d$
and thus
$\delta^{\prime}\delta^{\prime\prime}=d/\delta=\delta^2\ge\sqrt[3]{4}\approx 1.59>1$.
Consequently
$(\alpha\delta)^{\prime}(\alpha\delta)^{\prime\prime}=\alpha^{\prime}\alpha^{\prime\prime}\cdot\delta^{\prime}\delta^{\prime\prime}>\alpha^{\prime}\alpha^{\prime\prime}$
and $\alpha\delta>\alpha$, which means that $\mathbf{O}\neq\psi(\alpha)\in\mathcal{N}(\psi(\alpha\delta))$ and therefore 
$\alpha\delta\not\in\mathrm{Min}(\mathcal{O}_{L})$.
Similarly, the proof for $\alpha\bar{\delta}$.
\end{proof}

\newpage


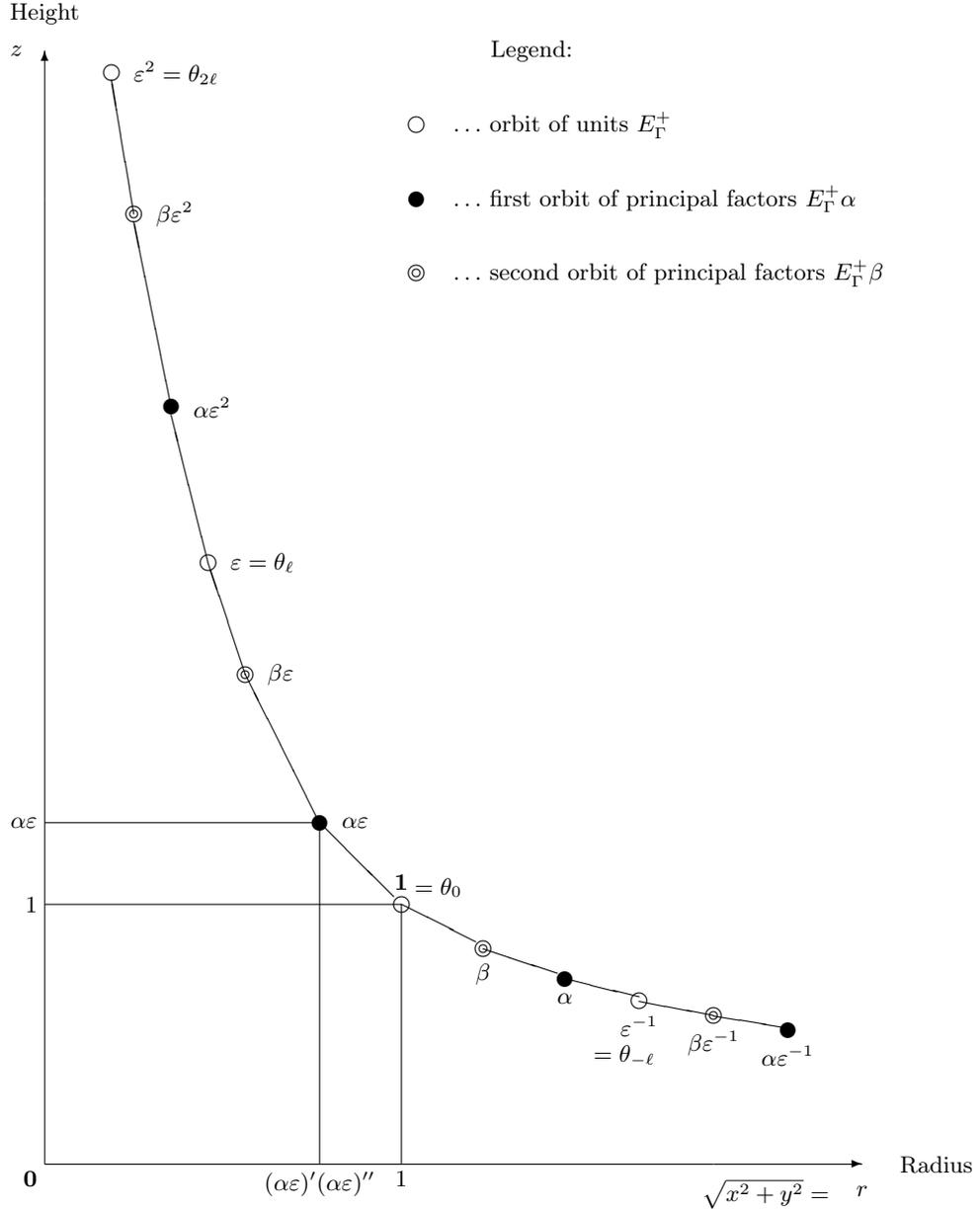
\begin{figure}[ht]
\caption{Chain $\Theta=(\theta_j)_{j\in\mathbb{Z}}$ of lattice minima in Minkowski signature space}
\label{fig:LatticeMinima}

{\small

\setlength{\unitlength}{1.0cm}
\begin{picture}(15,16.5)(-12,-11.5)




\put(-10,4.5){\makebox(0,0)[cc]{Height}}
\put(-10.3,4){\makebox(0,0)[rc]{\(z\)}}
\put(-10,2){\vector(0,1){2}}
\put(-10,2){\line(0,-1){13}}

\put(-10.1,-11.1){\makebox(0,0)[rt]{\(\mathbf{0}\)}}


\put(-10,-11){\line(1,0){9}}
\put(-1,-11){\vector(1,0){2}}
\put(-0.3,-11.2){\makebox(0,0)[ct]{\(\sqrt{x^2+y^2}=\)}}
\put(1,-11.3){\makebox(0,0)[ct]{\(r\)}}
\put(2,-11){\makebox(0,0)[cc]{Radius}}



\put(-9.1,3.7){\circle{0.2}}
\put(-8.8,3.7){\makebox(0,0)[lc]{\(\varepsilon^2=\theta_{2\ell}\)}}

\put(-8.8,1.8){\line(-1,6){0.3}}
\put(-8.8,1.8){\circle{0.1}}
\put(-8.8,1.8){\circle{0.2}}
\put(-8.5,1.8){\makebox(0,0)[lc]{\(\beta\varepsilon^2\)}}

\put(-8.3,-0.8){\line(-1,5){0.5}}
\put(-8.3,-0.8){\circle*{0.2}}
\put(-8,-0.8){\makebox(0,0)[lc]{\(\alpha\varepsilon^2\)}}

\put(-7.8,-2.9){\line(-1,4){0.5}}
\put(-7.8,-2.9){\circle{0.2}}
\put(-7.5,-2.9){\makebox(0,0)[lc]{\(\varepsilon=\theta_{\ell}\)}}

\put(-7.3,-4.4){\line(-1,3){0.5}}
\put(-7.3,-4.4){\circle{0.1}}
\put(-7.3,-4.4){\circle{0.2}}
\put(-7,-4.4){\makebox(0,0)[lc]{\(\beta\varepsilon\)}}

\put(-6.3,-6.4){\line(-1,2){1}}
\put(-6.3,-6.4){\circle*{0.2}}
\put(-6,-6.4){\makebox(0,0)[lc]{\(\alpha\varepsilon\)}}


\put(-6.3,-6.4){\line(-1,0){3.7}}
\put(-10.1,-6.4){\makebox(0,0)[rc]{\(\alpha\varepsilon\)}}
\put(-6.3,-6.4){\line(0,-1){4.6}}
\put(-6.3,-11.1){\makebox(0,0)[ct]{\((\alpha\varepsilon)^{\prime}(\alpha\varepsilon)^{\prime\prime}\)}}

\put(-6.3,-6.4){\line(1,-1){1}}

\put(-5.2,-7.5){\line(-1,0){4.8}}
\put(-10.1,-7.5){\makebox(0,0)[rc]{\(1\)}}
\put(-5.2,-7.5){\line(0,-1){3.5}}
\put(-5.2,-11.1){\makebox(0,0)[ct]{\(1\)}}


\put(-5.2,-7.3){\makebox(0,0)[cb]{\(\mathbf{1}\)}}
\put(-4.7,-7.4){\makebox(0,0)[cb]{\(=\theta_{0}\)}}
\put(-5.2,-7.5){\circle{0.2}}
\put(-5.2,-7.5){\line(2,-1){1}}

\put(-4.1,-8.3){\makebox(0,0)[ct]{\(\beta\)}}
\put(-4.1,-8.1){\circle{0.1}}
\put(-4.1,-8.1){\circle{0.2}}
\put(-4.1,-8.1){\line(3,-1){1}}

\put(-3,-8.7){\makebox(0,0)[ct]{\(\alpha\)}}
\put(-3,-8.5){\circle*{0.2}}
\put(-3,-8.5){\line(4,-1){1}}

\put(-2,-9){\makebox(0,0)[ct]{\(\varepsilon^{-1}\)}}
\put(-2.2,-9.4){\makebox(0,0)[ct]{\(=\theta_{-\ell}\)}}
\put(-2,-8.8){\circle{0.2}}
\put(-2,-8.8){\line(5,-1){1}}

\put(-1,-9.2){\makebox(0,0)[ct]{\(\beta\varepsilon^{-1}\)}}
\put(-1,-9){\circle{0.1}}
\put(-1,-9){\circle{0.2}}
\put(-1,-9){\line(6,-1){1}}

\put(0,-9.4){\makebox(0,0)[ct]{\(\alpha\varepsilon^{-1}\)}}
\put(0,-9.2){\circle*{0.2}}



\put(-4,4){\makebox(0,0)[lc]{Legend:}}
\put(-5,3){\circle{0.2}}
\put(-4.5,3){\makebox(0,0)[lc]{\(\ldots\) orbit of units \(E_{\Gamma}^+\)}}
\put(-5,2){\circle*{0.2}}
\put(-4.5,2){\makebox(0,0)[lc]{\(\ldots\) first orbit of principal factors \(E_{\Gamma}^+\alpha\)}}
\put(-5,1){\circle{0.1}}
\put(-5,1){\circle{0.2}}
\put(-4.5,1){\makebox(0,0)[lc]{\(\ldots\) second orbit of principal factors \(E_{\Gamma}^+\beta\)}}


\end{picture}

}

\end{figure}



\begin{definition}
\label{dfn:M0Fields}
A pure cubic field $L=\mathbb{Q}(\sqrt[3]{d})$ of principal factorization type $\beta$ is called an
\begin{itemize}
\item
$\mathrm{M}2$-\textit{field} if \quad
$\mathrm{Min}(\mathcal{O}_{L})\bigcap\Delta_{L/\mathbb{Q}}=E_{L}^+ \quad \dot{\bigcup} \quad E_{L}^+\alpha \quad \dot{\bigcup} \quad E_{L}^+\beta$,
\item
$\mathrm{M}1$-\textit{field} if \quad $\mathrm{Min}(\mathcal{O}_{L})\bigcap\Delta_{L/\mathbb{Q}}=E_{L}^+ \quad \dot{\bigcup} \quad E_{L}^+\alpha$ \quad
or \quad $E_{L}^+ \quad \dot{\bigcup} \quad E_{L}^+\beta$,
\item
$\mathrm{M}0$-\textit{field} if \quad $\mathrm{Min}(\mathcal{O}_{L})\bigcap\Delta_{L/\mathbb{Q}}=E_{L}^+$.
\qquad Here, $\beta$ denotes one of $\bar{\alpha}$, $\frac{\bar{\alpha}\delta}{d_1d_4d_5}$, $\frac{\bar{\alpha}\bar{\delta}}{d_1d_2d_4}$.
\end{itemize}
\end{definition}

\noindent
In Definition \ref{dfn:M0Fields},
which presents the mysterious $\mathrm{M}0$-\textit{fields}
as the central objects of our subsequent investigations,
because of their unpleasant impact on the classification problem
and corresponding serious defects in tables of cubic fields \cite{HCW1982},
we use the isomorphism
\begin{equation}
\label{eqn:Generators}
\mathcal{P}_{L}^G/\mathcal{P}_{\mathbb{Q}}\simeq\Delta_{L/\mathbb{Q}}/(E_{L}\cdot\mathbb{Q}^{\times}),
\end{equation}
induced by the principal ideal mapping $\iota:\,L^{\times}\to\mathcal{P}_{LL}$, $\alpha\mapsto (\alpha)=\alpha\mathcal{O}_{L}$,
with inverse image $\Delta_{L/\mathbb{Q}}:=\iota^{-1}(\mathcal{P}_{L}^G)$,
and we assume that the integral part $\Delta_{L/\mathbb{Q}}\cap\mathcal{O}_{L}$,
which always contains the radical group $\Delta:=\lbrace 1,\delta,\bar{\delta}\rbrace$,
is generated by the trivial principal factor $\delta$
and an additional non-trivial principal factor $\alpha$.
For the same reason as for
replacing the non-primitive square $\delta^2=\sqrt[3]{a^2b^4}=b\cdot\sqrt[3]{a^2b}=b\cdot\bar{\delta}$ by $\bar{\delta}:=\frac{\delta^2}{d_4d_5d_6}$
we also replace $\alpha^2$ by $\bar{\alpha}:=\frac{\alpha^2}{d_2d_5}$, as explained below by means of the canonical divisors.
Then we have
\begin{equation}
\label{eqn:Cosets}
\Delta_{L/\mathbb{Q}}\cap\mathcal{O}_{L}\simeq
\lbrace
\overbrace{1,\delta,\bar{\delta};}^{\text{trivial subgroup}}
\overbrace{\alpha,\frac{\alpha\delta}{d_2d_4d_5},\frac{\alpha\bar{\delta}}{d_1d_2d_5}}^{\text{first coset}};
\overbrace{\bar{\alpha},\frac{\bar{\alpha}\delta}{d_1d_4d_5},\frac{\bar{\alpha}\bar{\delta}}{d_1d_2d_4}}^{\text{second coset}}
\rbrace,
\end{equation}
represented by the norms (with abbreviations $ab^2=d_1d_2d_3d_4^2d_5^2d_6^2$, $a^2b=d_1^2d_2^2d_3^2d_4d_5d_6$)
\begin{equation}
\label{eqn:Norms}
\lbrace
\overbrace{1,ab^2,a^2b}^{\text{trivial subgroup}};
\overbrace{d_1d_2^2d_4d_5^2,d_1^2d_3d_5d_6^2,d_2d_3^2d_4^2d_6}^{\text{first coset}};
\overbrace{d_1^2d_2d_4^2d_5,d_2^2d_3d_4d_6^2,d_1d_3^2d_5^2d_6}^{\text{second coset}}
\rbrace.
\end{equation}


\begin{theorem}
\label{thm:WilliamsErrors}
Among the $12\,220$ pure cubic fields $L=\mathbb{Q}(\sqrt[3]{d})$
with normalized radicands in the range $2\le d<15\,000$,
there occur more $\mathrm{M}0$-fields than the $16$ cases
listed by H. C. Williams \cite[\S\ 6, Tbl. 2, p. 273]{HCW1982},
\begin{equation}
\label{eqn:Correct}
2,455,833,850,1078,1235,1573,3857,4901,6061,6358,8294,8959,12121,12818,14801.
\end{equation}
The five missing radicands are:
\begin{equation}
\label{eqn:Missing}
1\,430,\ 6\,370, \ 9\,922,\ 11\,284,\ 12\,673.
\end{equation}
So there are precisely $21$ cases of $\mathrm{M}0$-fields in this range.
\end{theorem}



\begin{lemma}
\label{lem:Periodicity}
If the fundamental unit $\varepsilon$ is the $\ell$th lattice minimum,
counted from the trivial unit $1$ in the direction of increasing height,
then the norms of lattice minima are periodic with primitive period length $\ell$, that is,
\begin{equation}
\label{eqn:periodicity}
(\forall\,0\le j\le\ell-1)\,(\forall\,n\in\mathbb{Z})\,N_{L/\mathbb{Q}}(\theta_{j+n\cdot\ell})=N_{L/\mathbb{Q}}(\theta_{j}).
\end{equation}
\end{lemma}

\begin{proof}
Let $\varepsilon>1$ be the normpositive fundamental unit bigger than the trivial unit $1$ of $L$.
Then $0<\varepsilon^{-1}<1$ is the inverse normpositive fundamental unit of $L$.
Due to the decomposition
\begin{equation}
\label{eqn:Orbits}
\Theta=(\theta_j)_{j\in\mathbb{Z}}=((\theta_{j+n\cdot\ell})_{n\in\mathbb{Z}})_{0\le j<\ell},
\quad \text{ respectively } \quad
\mathrm{Min}(\mathcal{O}_{L})=\bigcup_{j=0}^{\ell-1}\,E_{L}^+\cdot\theta_j
\end{equation}
of the chain $\Theta$, respectively of the set  $\mathrm{Min}(\mathcal{O}_{L})$,
where $\theta_{n\cdot\ell}=\varepsilon^n$ for all $n\in\mathbb{Z}$,
into orbits under the action of $E_{L}^+=\lbrace\varepsilon^n\mid n\in\mathbb{Z}\rbrace$
with representatives $1\le\theta_{j}<\varepsilon$, $0\le j<\ell$, in the first primitive period,
visualized impressively in Figure \ref{fig:LatticeMinima},
we have
\begin{equation}
\label{eqn:associates}
(\forall\,0\le j\le\ell-1)\,(\forall\,n\in\mathbb{Z})\,\theta_{j+n\cdot\ell}=\varepsilon^n\cdot\theta_{j},
\end{equation}
and thus
$N_{L/\mathbb{Q}}(\theta_{j+n\cdot\ell})
=N_{L/\mathbb{Q}}(\varepsilon^n\cdot\theta_{j})
=N_{L/\mathbb{Q}}(\varepsilon)^n\cdot N_{L/\mathbb{Q}}(\theta_{j})
=1\cdot N_{L/\mathbb{Q}}(\theta_{j})
=N_{L/\mathbb{Q}}(\theta_{j})$.
\end{proof}



\section{Necessary and sufficient conditions for minimal principal factors}
\label{ss:PrincipalFactorMinima}

\noindent
We now state the main theorem on principal factors among the lattice minima.

\begin{theorem}
\label{thm:PrincipalFactorMinima}
Let $L=\mathbb{Q}(\sqrt[3]{d})$ be a pure cubic field
of principal factorization type $\beta$
with normalized cube-free radicand $d=ab^2>1$.
Suppose that $\alpha\in\mathcal{O}_{L}\setminus E_{L}$ is
generator of a primitive ambiguous principal ideal
$(\alpha)\in\mathcal{P}_{L}^G/{P}_{\mathbb{Q}}$ of $L$
with norm $n=N_{L/\mathbb{Q}}(\alpha)=3^v\cdot d_1d_2^2d_4d_5^2$,
where $v\ge 1$ at most for $d\equiv\pm 2,\pm 4\,(\mathrm{mod}\,9)$,
and that $\gamma=\sqrt[3]{ab^2}/d_2d_4d_5>1$ and $\bar{\gamma}=\sqrt[3]{a^2b}/d_1d_2d_5>1$.
Then the criteria for the occurrence of $\alpha$ among the lattice minima
of the chain $\Theta$ of the maximal order $\mathcal{O}_{L}$,
respectively $\Phi$ of the non-maximal order $\mathcal{O}_{L,0}$
with conductor $\mathfrak{l}^\sigma\mathfrak{l}$, where $3\mathcal{O}_{L}=\mathfrak{l}^\sigma\mathfrak{l}^2$
\cite{Dk},
if $d\equiv\pm 1\,(\mathrm{mod}\,9)$,
can be partitioned in the following way:

\begin{itemize}

\item
Unconditional criteria:

\begin{enumerate}
\item
If $L$ is of species $1\mathrm{a}$, $3\mid d$, then $\alpha\in\mathrm{Min}(\mathcal{O}_{L})$.
\item
If $L$ is of species $1\mathrm{b}$, $d\equiv\pm 2,\pm 4\,(\mathrm{mod}\,9)$,
and $v=0$, then $\alpha\in\mathrm{Min}(\mathcal{O}_{L})$.
\item
If $L$ is of species $2$, $d\equiv\pm 1\,(\mathrm{mod}\,9)$, then $\alpha\in\mathrm{Min}(\mathcal{O}_{L,0})$.
\end{enumerate}

\item
Conditional criteria in dependence on $u_1\equiv d_1d_3d_4d_5\,(\mathrm{mod}\,3)$ and $u_2\equiv d_1d_2d_4d_6\,(\mathrm{mod}\,3)$:

\begin{enumerate}
\item
If $L$ is of species $1\mathrm{b}$, $d\equiv\pm 2,\pm 4\,(\mathrm{mod}\,9)$,
and $v=1$, or $L$ is of species $2$, $d\equiv\pm 1\,(\mathrm{mod}\,9)$,
let two critical bivariate polynomials be defined by
\begin{equation}
\label{eqn:Polynomial}
\begin{aligned}
P_2(X,Y) &:= X^2+Y^2-XY-X-Y+1\in\mathbb{Z}\lbrack X,Y\rbrack, \\
P_4(X,Y) &:= X^4-X^3+X^2Y-8X^2+XY+Y^2\in\mathbb{Z}\lbrack X,Y\rbrack.
\end{aligned}
\end{equation}
Then the following necessary and sufficient criterion holds:
\begin{equation}
\label{eqn:Escalatory3}
\begin{aligned}
\alpha\not\in\mathrm{Min}(\mathcal{O}_{L})
&\Longleftrightarrow
(u_1,u_2)\neq (1,1) \text{ and } P_2(u_1L,u_2\bar{\gamma})<9 \\
&\Longleftrightarrow
(u_1,u_2)\neq (1,1) \text{ and } P_4(u_1\gamma,-u_1u_2y)<0 \\
&\Longleftrightarrow
(u_1,u_2)\neq (1,1) \text{ and } P_4(u_2\bar{\gamma},-u_1u_2y)<0.
\end{aligned}
\end{equation}
For $(u_1,u_2)\neq (1,1)$, a coarse sufficient, but not necessary, condition is given by:
\begin{equation}
\label{eqn:InverseCoarse3}
\max\left(\frac{\gamma}{B(u_2)},\frac{\bar{\gamma}}{B(u_1)}\right)\ge 1
\Longrightarrow
\alpha\in\mathrm{Min}(\mathcal{O}_{L}),
\end{equation}
where the bound is defined by
\begin{equation}
\label{eqn:Bound3}
B(u):=
\begin{cases}
\sqrt{6}\approx 2.44948974278318 & \text{ if } u=-1, \\
2                                & \text{ if } u=1.
\end{cases}
\end{equation}
\item
If $L$ is of species $1\mathrm{b}$, $d\equiv\pm 2,\pm 4\,(\mathrm{mod}\,9)$,
and $v=2$,
let a critical bound be defined by
\begin{equation}
\label{eqn:Bound9}
C(u):=
\begin{cases}
\frac{1}{2}(-1+\sqrt{33})\approx 2.37228132326901 & \text{ if } u=1, \\
2                                                 & \text{ if } u=-1.
\end{cases}
\end{equation}
Then the following necessary and sufficient criterion holds:
\begin{equation}
\label{eqn:Escalatory9}
\alpha\not\in\mathrm{Min}(\mathcal{O}_{L})
\Longleftrightarrow
\min\left(\frac{\gamma}{C(u_1)},\frac{\bar{\gamma}}{C(u_2)}\right)<1.
\end{equation}
\end{enumerate}

\end{itemize}

\end{theorem}


\begin{proof}
The major part of the proof is due to Williams.
However, it is scattered among several papers
\cite{HCW1980,HCW1981,HCW1982},
and some cases have never been formulated as necessary and sufficient criteria.
Generally, let $\alpha\in\mathcal{O}_{L}$ be a principal factor with norm
$n=N_{L/\mathbb{Q}}(\alpha)=3^v\cdot d_1d_2^2d_4d_5^2$,
where $v\in\lbrace 0,1,2\rbrace$
and $v\ge 1$ at most for $d\equiv\pm 2,\pm 4\,(\mathrm{mod}\,9)$.

\begin{itemize}

\item
Concerning the unconditional criteria:

\begin{enumerate}
\item
The claim that generally $\alpha\in\mathrm{Min}(\mathcal{O}_{L})$
for $d\equiv 0,\pm 3\,(\mathrm{mod}\,9)$ (whence $v=0$)
is proved in
\cite[\S\ 4, Thm. 2, p. 1427]{HCW1980}
and again in
\cite[\S\ 5, Thm. 5.1(i), p. 643]{HCW1981}.
\item
$\alpha\in\mathrm{Min}(\mathcal{O}_{L})$
for $d\equiv\pm 2,\pm 4\,(\mathrm{mod}\,9)$ with $v=0$
is also proven in
\cite[Thm. 2]{HCW1980}.
\item
The statement that $\alpha\in\mathrm{Min}(\mathcal{O}_{L,0})$
for $d\equiv\pm 1\,(\mathrm{mod}\,9)$ (and hence $v=0$)
is due to ourselves, and provides considerable computational simplification,
as Theorem \ref{thm:Algorithm} will show.
For fields of the second species,
$(1,\delta,\bar{\delta})$ is not an integral basis of the maximal order $\mathcal{O}_{L}$,
but it is a basis of the non-maximal order
$\mathcal{O}_{L,0}=\mathbb{Z}\oplus\mathbb{Z}\delta\oplus\mathbb{Z}\bar{\delta}$
with conductor $\mathfrak{l}^\sigma\mathfrak{l}$, where $3\mathcal{O}_{L}=\mathfrak{l}^\sigma\mathfrak{l}^2$.
The proof in 
\cite[\S\ 4, Thm. 2, p. 1427]{HCW1980}
is generally valid for the order $\mathbb{Z}\oplus\mathbb{Z}\delta\oplus\mathbb{Z}\bar{\delta}$
and does not use the incongruence $d\not\equiv\pm 1\,(\mathrm{mod}\,9)$.
Thus it also holds for $d\equiv\pm 1\,(\mathrm{mod}\,9)$.
\end{enumerate}

\item
Concerning the conditional criteria for
either $d\equiv\pm 2,\pm 4\,(\mathrm{mod}\,9)$ with $v\ge 1$
or the maximal order in the case $d\equiv\pm 1\,(\mathrm{mod}\,9)$,
\cite[\S\ 3, Thm. 3.4, p. 638]{HCW1981}
establishes a diophantine criterion for the existence of a non-trivial lattice point
within the norm cylinder of an algebraic integer with principal factor norm.
In \cite[\S\ 4, Lem. 4.1, p. 639]{HCW1981},
the possible solutions of
this critical system of diophantine ternary quadratic inequalities
are narrowed down generally.

\begin{enumerate}
\item
For either $d\equiv\pm 2,\pm 4\,(\mathrm{mod}\,9)$ with $v=1$
or the maximal order in the case $d\equiv\pm 1\,(\mathrm{mod}\,9)$,
it is shown in \cite[\S\ 4, Lem. 4.2, p. 640]{HCW1981}
that the diophantine criterion has a unique solution in dependence on $(u_1,u_2)$,
except for $(u_1,u_2)=(1,1)$, where $\alpha\in\mathrm{Min}(\mathcal{O}_{L})$ turns out generally.
The final conclusion is given in the later paper
\cite[\S\ 4, Thm. 4.1, p. 268]{HCW1982}
in terms of our quadratic polynomial $P_2(X,Y)$.
Our transformation in terms of the fourth degree polynomial $P_4(X,Y)$ is new and
permits the deduction of a coarse sufficient condition for the converse statement
in formulas \eqref{eqn:InverseCoarse3} and \eqref{eqn:Bound3}
by investigating the zero locus of $P_4(X,Y)$ in the $XY$-plane.
An even coarser sufficient condition is given in
\cite[\S\ 5, Thm. 5.1(ii)--(iii), p. 643]{HCW1981}
by generally taking the bigger bound $\sqrt{6}>2$.
\item
Finally, for $d\equiv\pm 2,\pm 4\,(\mathrm{mod}\,9)$ with $v=2$,
a few solutions of the diophantine criterion
are found in \cite[\S\ 4, Lem. 4.3, p. 642]{HCW1981} in dependence on $(u_1,u_2)$,
but no concluding theorem is stated.
We proved that the solution in dependence on $(u_1,u_2)$ is in fact unique
for each of the normalized radicals $\gamma$ and $\bar{\gamma}$,
which leads to the necessary and sufficient criterion
in formulas \eqref{eqn:Bound9} and \eqref{eqn:Escalatory9}.
A coarse sufficient condition for the converse statement is given in
\cite[\S\ 5, Thm. 5.1(vi), p. 643]{HCW1981}
by generally taking the bigger bound $\frac{1}{2}(-1+\sqrt{33})>2$ 
\end{enumerate}
\end{itemize}
\end{proof}

\newpage


\begin{figure}[ht]
\caption{Zero locus of $P_4(X,Y)$}
\label{fig:ZeroLocus}
\includegraphics[width=14cm]{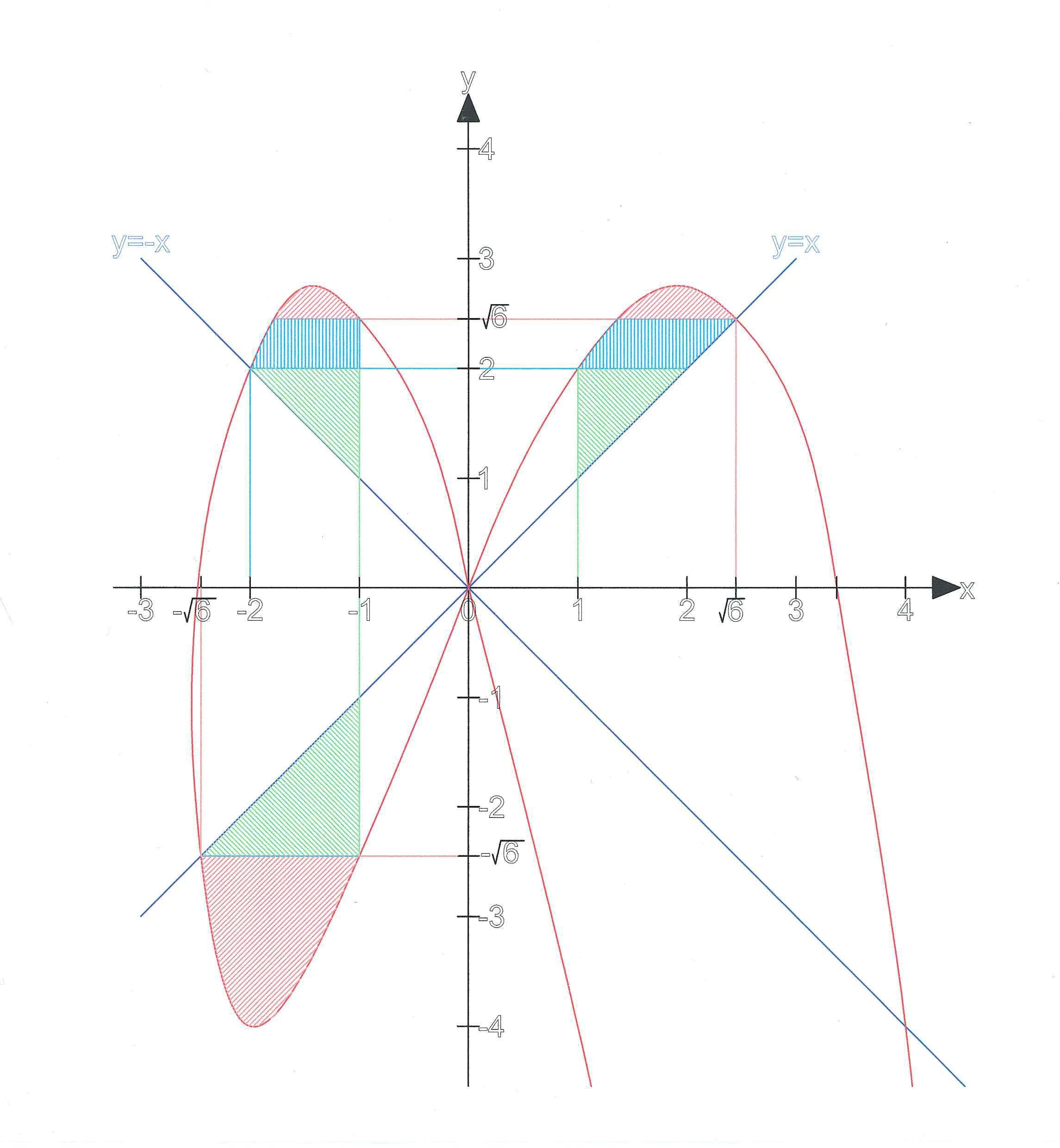}
\end{figure}

\noindent
In Figure \ref{fig:ZeroLocus},
the upper part $Y\ge 4$ of the zero locus 
of the bivariate polynomial $P_4(X,Y)\in\mathbb{Z}\lbrack X,Y\rbrack$
is plotted. This is the part which is relevant for deciding
whether a principal factor whose norm is not divisible by $9$ is a lattice minimum or not,
because in Equation \eqref{eqn:Escalatory3} of Theorem \ref{thm:PrincipalFactorMinima} 
the conditions $P_4(u_1\gamma,-u_1u_2y)<0$ and $P_4(u_2\bar{\gamma},-u_1u_2y)<0$ must be checked, both for
$(u_1,u_2)\neq (1,1)$, $\gamma>1$, $\bar{\gamma}>1$ and $y=\gamma\bar{\gamma}\ge\max(\gamma,\bar{\gamma})$.
Consequently, the quadrant $X>0$, $Y<0$, where the zero locus reaches down to $Y=-16$, does not concern the decision.
In the green triangles $Y\le\sqrt{6}$, respectively $Y\ge 2$, the condition holds automatically,
in the blue regions, only the left, and
in the red regions the left and right inequality must be tested.



\begin{corollary}
\label{cor:PrincipalFactorMinima}
Under the assumptions and notations of Theorem \ref{thm:PrincipalFactorMinima},
a further coarse sufficient, but not necessary condition, is given by:
\begin{equation}
\label{eqn:Coarse3}
(u_1,u_2)\neq (1,1) \text{ and } y\le B(-u_1u_2)
\Longrightarrow
\alpha\not\in\mathrm{Min}(\mathcal{O}_{L}),
\end{equation}
for either $d\equiv\pm 2,\pm 4\,(\mathrm{mod}\,9)$ with $v=1$ or $d\equiv\pm 1\,(\mathrm{mod}\,9)$.
\end{corollary}

\begin{proof}
This also follows from investigating the zero locus of $P_4(X,Y)$ in the $XY$-plane.
\end{proof}



\section{Classification algorithm}
\label{ss:Algorithm}

\noindent
We continue with another main theorem on the classification
of pure cubic fields into \textit{principal factorization types}
\cite[\S\ 2.1]{AMITA}
with the aid of Voronoi's algorithm.
The decisive innovation in contrast to previous classification algorithms
is the use of a non-maximal order for species $2$.

\begin{theorem}
\label{thm:Algorithm}
Let $L=\mathbb{Q}(\sqrt[3]{d})$ be a pure cubic field
with normalized cube-free radicand $d\ge 2$, ramification invariant $R$,
according to equation \eqref{eqn:Ramification},
and subfield unit index $Q$,
according to \cite[\S\ 2.1]{AMITA}.
Denote the chain of lattice minima of the maximal order $\mathcal{O}_{L}$
by $\Theta=(\theta_j)_{j\in\mathbb{Z}}$
and its primitive period length by $\ell\ge 1$.
Then the following necessary and sufficient criteria
determine the principal factorization type of $L$
in dependence on the Dedekind species of the radicand $d$.
\begin{enumerate}
\item
If $L$ belongs to species $1\mathrm{a}$, $d\equiv 0,\pm 3\,(\mathrm{mod}\,9)$, then $L$ is of
\begin{enumerate}
\item
type $\alpha$ $\Longleftrightarrow$ $Q=1$, 
\item
type $\beta$ $\Longleftrightarrow$
$(\exists\,1\le j\le\ell-1)\,N_{L/\mathbb{Q}}(\theta_j)\mid R^2$,
\item
type $\gamma$ $\Longleftrightarrow$
$(\forall\,1\le j\le\ell-1)\,N_{L/\mathbb{Q}}(\theta_j)\nmid R^2$ and $Q=3$.
\end{enumerate}
\item
If $L$ belongs to species $1\mathrm{b}$, $d\equiv\pm 2,\pm 4\,(\mathrm{mod}\,9)$, then $L$ is of
\begin{enumerate}
\item
type $\alpha$ $\Longleftrightarrow$ $Q=1$,
\item
type $\beta$ $\Longleftrightarrow$
either $(\exists\,1\le j\le\ell-1)\,N_{L/\mathbb{Q}}(\theta_j)\mid R^2$ or $Q=3$.
\item
For this species, $L$ can never be of type $\gamma$.
\end{enumerate}
\item
If $L$ belongs to species $2$, $d\equiv\pm 1\,(\mathrm{mod}\,9)$,
let $\Phi=(\phi_j)_{j\in\mathbb{Z}}$ be the chain of lattice minima of the
non-maximal order $\mathcal{O}_{L,0}$
with conductor $\mathfrak{l}^\sigma\mathfrak{l}$, where $3\mathcal{O}_{L}=\mathfrak{l}^\sigma\mathfrak{l}^2$
\cite{Dk},
and $\ell_0\ge 1$ its primitive period length.
Then $L$ is of
\begin{enumerate}
\item
type $\alpha$ $\Longleftrightarrow$ $Q=1$,
\item
type $\beta$ $\Longleftrightarrow$ 
$(\exists\,1\le j\le\ell_0-1)\,N_{L/\mathbb{Q}}(\phi_j)\mid R^2$,
\item
type $\gamma$ $\Longleftrightarrow$ 
$(\forall\,1\le j\le\ell_0-1)\,N_{L/\mathbb{Q}}(\phi_j)\nmid R^2$ and $Q=3$.
\end{enumerate}
\end{enumerate}
\end{theorem}

\begin{remark}
\label{rmk:Algorithm}
This remarkable algorithm deserves several remarks.
\begin{enumerate}
\item
Our progressive innovation to use the non-maximal order
for the guaranteed detection of an absolute principal factor
is an incredibly powerful and easily implementable technique
which circumvents the error prone method of Williams in \cite[\S\ 4, pp. 268--271]{HCW1982}.
\item
Actually, we have used this algorithm to achieve
the extensive classification of all $827\,600$
fields with $d<10^6$, as described in \cite[Exm. 2.1]{AMITA}.
For more detailed statistics see Table \ref{tbl:Classification},
where column $B=15\,000$ is included with corrected results for \cite[\S\ 6, p. 272]{HCW1982}.
\item
For item 2.(b) of Theorem \ref{thm:Algorithm},
$Q=3$ alone would be sufficient,
but the determination of $Q$ requires the fundamental unit $\varepsilon=\theta_{\ell}$
at the end of the full period,
whereas usually a $\theta_j$ with $N_{L/\mathbb{Q}}(\theta_j)\mid R^2$
has a subscript $1\le j<\ell$ of approximate magnitude $\ell/3$ or $2\ell/3$,
and thus admits an earlier termination of the algorithm at a third or two thirds of the period.
\end{enumerate}
\end{remark}


\renewcommand{\arraystretch}{1.0}

\begin{table}[ht]
\caption{Distribution of principal factorization types for $2\le d\le B$}
\label{tbl:Classification}
\begin{center}
\begin{tabular}{|c||r|r|r|r|r|r|r|}
\hline
 bound $B$          & $10$ & $100$ & $1\,000$ & $10\,000$ & $15\,000$ & $100\,000$ & $1\,000\,000$ \\
\hline
 $\#$ type $\alpha$ &  $1$ &  $19$ &    $182$ &  $1\,753$ &  $2\,606$ &  $16\,935$ &    $163\,527$ \\
 $\#$ type $\beta$  &  $4$ &  $49$ &    $556$ &  $5\,988$ &  $9\,058$ &  $62\,068$ &    $635\,463$ \\
 $\#$ type $\gamma$ &  $1$ &   $6$ &     $50$ &     $381$ &     $556$ &   $3\,261$ &     $28\,610$ \\
\hline
 $\#$ total         &  $6$ &  $74$ &    $788$ &  $8\,122$ & $12\,220$ &  $82\,264$ &    $827\,600$ \\
\hline
\end{tabular}
\end{center}
\end{table}


\begin{proof}
The equivalence of type $\alpha$ with a subfield unit index $Q=1$
is true independently of the Dedekind species,
according to \cite[Eqn. (5) in Rmk. 2.1]{AMITA}.
For the other two types $\beta$ and $\gamma$, where $Q=3$ for both, we distinguish the species.
\begin{enumerate}
\item
For species $1\mathrm{a}$, $d\equiv 0,\pm 3\,(\mathrm{mod}\,9)$,
the unconditional criterion 1 in Theorem \ref{thm:PrincipalFactorMinima}
proves that a non-unit $\alpha$ with norm $n=N_{L/\mathbb{Q}}(\alpha)$ dividing $R^2$
must occur as a lattice minimum $\alpha=\theta_j$ in the chain $\Theta$ of the maximal order $\mathcal{O}_{L}$.
Thus the occurrence of such a $\theta_j$ is equivalent with type $\beta$.
The lack of such a $\theta_j$ implies type $\alpha$ or $\gamma$
and type $\alpha$ must be discouraged by $Q=3$.
\item
A necessary condition for type $\gamma$,
that is, the occurrence of a unit $Z\in E_{\mathit{k}}$ such that $N_{\mathit{k}/\mathit{k}_0}(Z)=\zeta_3$,
is that the conductor $f$ of $\mathit{k}/\mathit{k}_0$ is divisible only by $3$ or primes $\ell\equiv\pm 1\,(\mathrm{mod}\,9)$.
For species $1\mathrm{b}$, $d\equiv\pm 2,\pm 4\,(\mathrm{mod}\,9)$,
there must exist a prime divisor $\ell\equiv\pm 2,\pm 4\,(\mathrm{mod}\,9)$ of $f$
and type $\gamma$ is impossible.
Therefore, type $\beta$ is equivalent with $Q=3$,
and only for accelerating the algorithm it is worth while to check
the possible occurrence of a lattice minimum with norm dividing $R^2$.
\item
For species $2$, $d\equiv\pm 1\,(\mathrm{mod}\,9)$,
the unconditional criterion 3 in Theorem \ref{thm:PrincipalFactorMinima}
shows that a non-unit $\alpha$ with norm $n=N_{L/\mathbb{Q}}(\alpha)$ dividing $R^2$
must occur as a lattice minimum $\alpha=\phi_j$ in the chain $\Phi$ of the non-maximal order $\mathcal{O}_{L,0}$.
Therefore the occurrence of such a $\phi_j$ is equivalent with type $\beta$.
The lack of such a $\phi_j$ enforces type $\alpha$ or $\gamma$
and type $\alpha$ must be eliminated by $Q=3$.
(Note that $\alpha$ is coprime to the conductor $\mathfrak{l}^\sigma\mathfrak{l}$ \cite{Dk}.)\qedhere
\end{enumerate}
\end{proof}



\section{Explicit criteria for M0-fields in rational integers}
\label{ss:M0Fields}

\noindent
It is useful to specialize the general Theorem \ref{thm:PrincipalFactorMinima}
to situations, where the occurrence of a principal factor among the lattice minima
can be characterized in terms of the canonical divisors $d_1,\ldots,d_6$.
The most convenient situation appears for a squarefree radicand
$d=d_1d_2d_3$, where $d_4=d_5=d_6=1$, a priori.

\begin{theorem}
\label{thm:SquareFreeSpec2}
Let the squarefree radicand $d=d_1d_2d_3$ be of second species, $d\equiv\pm 1\,(\mathrm{mod}\,9)$,
and assume there exists a principal factor $\alpha\in\mathcal{O}_{L}$
with norm $n=d_1d_2^2$, minimal in the first coset $\lbrace d_1d_2^2,d_1^2d_3,d_2d_3^2\rbrace$, that is
\begin{equation}
\label{eqn:M0FieldSpec2Coarse}
d_2^2<d_1d_3, \quad d_1d_2<d_3^2.
\end{equation}
\begin{itemize}
\item
If $d_1^2<d_2d_3$, then $\bar{n}=d_1^2d_2$ is minimal in the second coset $\lbrace d_1^2d_2,d_2^2d_3,d_1d_3^2\rbrace$,
and $L$ is an $\mathrm{M}0$-field (neither $\alpha\in\mathrm{Min}(\mathcal{O}_{L})$ nor $\bar{\alpha}\in\mathrm{Min}(\mathcal{O}_{L})$), if
\begin{equation}
\label{eqn:M0FieldSpec2Var1}
\begin{aligned}
\text{either } & d_1\equiv d_2\equiv -d_3\,(\mathrm{mod}\,3), \quad d_3\le 2\min(d_1,d_2) \\
\text{or }     & d_1\equiv -d_2\equiv d_3\,(\mathrm{mod}\,3), \quad d_3\le\min(\sqrt{6}d_1,2d_2) \\
\text{or }     & -d_1\equiv d_3\equiv d_2\,(\mathrm{mod}\,3), \quad d_3\le\min(2d_1,\sqrt{6}d_2).
\end{aligned}
\end{equation}
\item
If $d_2d_3<d_1^2$, then $\bar{n}=d_2^2d_3$ is minimal in the second coset $\lbrace d_2^2d_3,d_1d_3^2,d_1^2d_2\rbrace$,
and $L$ is an $\mathrm{M}0$-field (neither $\alpha\in\mathrm{Min}(\mathcal{O}_{L})$ nor $\bar{\alpha}\delta/d_1\in\mathrm{Min}(\mathcal{O}_{L})$), if
\begin{equation}
\label{eqn:M0FieldSpec2Var2}
\begin{aligned}
\text{either } & d_1\equiv d_2\equiv -d_3\,(\mathrm{mod}\,3), \quad d_1\le\sqrt{6}d_2, \quad d_3\le 2d_2 \\
\text{or }     & d_1\equiv -d_2\equiv d_3\,(\mathrm{mod}\,3), \quad \max(d_1,d_3)\le 2d_2 \\
\text{or }     & -d_1\equiv d_3\equiv d_2\,(\mathrm{mod}\,3), \quad d_1\le 2d_2, \quad d_3\le\sqrt{6}d_2.
\end{aligned}
\end{equation}
\end{itemize}
\end{theorem}


\begin{proof}
The claim concerns both non-trivial cosets of principal factors,
the first coset of $\alpha$ with norm $n=d_1d_2^2$
and the second coset of $\bar{\alpha}=\alpha^2/d_2$, respectively $\bar{\alpha}\delta/d_1$, with norm $\bar{n}$.

First, we consider the coset of $\alpha$.
Here, we have the congruence invariants
$u_1\equiv d_1d_3\,(\mathrm{mod}\,3)$, $u_2\equiv d_1d_2\,(\mathrm{mod}\,3)$,
the normalized radicals
$\gamma=\sqrt[3]{d_1d_2d_3}/d_2>1$, $\bar{\gamma}=\sqrt[3]{d_1^2d_2^2d_3^2}/d_1d_2>1$,
and their product
$y=\gamma\bar{\gamma}=(\sqrt[3]{d_1d_2d_3}/d_2)(\sqrt[3]{d_1^2d_2^2d_3^2}/d_1d_2)=d_3/d_2$.
The minimality of $n$ in its coset yields relations between the magnitude of the canonical divisors,
$d_2^2<d_1d_3$ and $d_1d_2<d_3^2$, that is, formula \eqref{eqn:M0FieldSpec2Coarse}.

We exploit the coarse sufficient condition in Corollary \ref{cor:PrincipalFactorMinima}:
$y\le B(-u_1u_2)$ $\Longrightarrow$ \\
$\alpha\not\in\mathrm{Min}(\mathcal{O}_{L})$,
that is, $d_3\le\sqrt{6}\cdot d_2$ if $u_1=u_2=-1$, and $d_3\le 2\cdot d_2$ otherwise.
The connection between the congruence invariants and the residue class of the canonical divisors is given by
the forbidden case $d_1\equiv d_2\equiv d_3\,(\mathrm{mod}\,3)$ $\Longleftrightarrow$ $(u_1,u_2)=(1,1)$,
and the admissible cases \\
$d_1\equiv d_2\equiv -d_3\,(\mathrm{mod}\,3)$ $\Longleftrightarrow$ $(u_1,u_2)=(-1,1)$, \\
$d_1\equiv -d_2\equiv d_3\,(\mathrm{mod}\,3)$ $\Longleftrightarrow$ $(u_1,u_2)=(1,-1)$, \\
$-d_1\equiv d_2\equiv d_3\,(\mathrm{mod}\,3)$ $\Longleftrightarrow$ $(u_1,u_2)=(-1,-1)$.

For the second coset, we have to split the investigation.
\begin{itemize}
\item
If $d_1^2<d_2d_3$, then the minimal norm is $\bar{n}=d_1^2d_2$.
with new canonical invariants $\bar{n}=c_1c_2^2$, where
$c_1:=d_2$ and $c_2:=d_1$ are twisted, whereas $c_3=d_3$ remains fixed.

The connection between the congruence invariants and the residue class of the canonical divisors is given by \\
$d_1\equiv d_2\equiv -d_3\,(\mathrm{mod}\,3)$ $\Longleftrightarrow$ $c_1\equiv c_2\equiv -c_3\,(\mathrm{mod}\,3)$ $\Longleftrightarrow$ $(u_1,u_2)=(-1,1)$, \\
$d_1\equiv -d_2\equiv d_3\,(\mathrm{mod}\,3)$ $\Longleftrightarrow$ $-c_1\equiv c_2\equiv c_3\,(\mathrm{mod}\,3)$ $\Longleftrightarrow$ $(u_1,u_2)=(-1,-1)$, \\
$-d_1\equiv d_2\equiv d_3\,(\mathrm{mod}\,3)$ $\Longleftrightarrow$ $c_1\equiv -c_2\equiv c_3\,(\mathrm{mod}\,3)$ $\Longleftrightarrow$  $(u_1,u_2)=(1,-1)$.

Again, we employ the coarse sufficient condition in Corollary \ref{cor:PrincipalFactorMinima}:
$y\le B(-u_1u_2)$ $\Longrightarrow$ $\bar{\alpha}\not\in\mathrm{Min}(\mathcal{O}_{L})$,
that is, $d_3=c_3\le\sqrt{6}\cdot c_2=\sqrt{6}\cdot d_1$ if $u_1=u_2=-1$, and $d_3=c_3\le 2\cdot c_2=2\cdot d_1$ otherwise.
\item
If $d_2d_3<d_1^2$, then the minimal norm is $\bar{n}=d_2^2d_3$.
with new canonical invariants $\bar{n}=c_1c_2^2$, where
$c_1:=d_3$ and $c_3:=d_1$ are twisted, whereas $c_2=d_2$ remains fixed.

The connection between the congruence invariants and the residue class of the canonical divisors is given by \\
$d_1\equiv d_2\equiv -d_3\,(\mathrm{mod}\,3)$ $\Longleftrightarrow$ $-c_1\equiv c_2\equiv c_3\,(\mathrm{mod}\,3)$ $\Longleftrightarrow$ $(u_1,u_2)=(-1,-1)$, \\
$d_1\equiv -d_2\equiv d_3\,(\mathrm{mod}\,3)$ $\Longleftrightarrow$ $c_1\equiv -c_2\equiv c_3\,(\mathrm{mod}\,3)$ $\Longleftrightarrow$ $(u_1,u_2)=(1,-1)$, \\
$-d_1\equiv d_2\equiv d_3\,(\mathrm{mod}\,3)$ $\Longleftrightarrow$ $c_1\equiv c_2\equiv -c_3\,(\mathrm{mod}\,3)$ $\Longleftrightarrow$  $(u_1,u_2)=(-1,1)$.

Again, we employ the coarse sufficient condition in Corollary \ref{cor:PrincipalFactorMinima}:
$y\le B(-u_1u_2)$ $\Longrightarrow$ $\bar{\alpha}\delta/d_1\not\in\mathrm{Min}(\mathcal{O}_{L})$,
that is, $d_1=c_3\le\sqrt{6}\cdot c_2=\sqrt{6}\cdot d_2$ if $u_1=u_2=-1$, and $d_1=c_3\le 2\cdot c_2=2\cdot d_2$ otherwise.
\end{itemize}
Finally we collect all required inequalities for the first and second non-trivial coset,
and we must make sure that not $u_1=u_2=1$,
which is the case if not $d_1\equiv d_2\equiv d_3\,(\mathrm{mod}\,3)$.
\end{proof}

\noindent
Theorem \ref{thm:SquareFreeSpec2} gives rise to the following hypothesis,
since the assumptions for the three positive integers $d_1,d_2,d_3$
in form of simple inequalities and simple congruences modulo $3$
seem to be satisfiable even by infinitely many triples $(d_1,d_2,d_3)\in\mathbb{P}^3$ of prime numbers.

\begin{conjecture}
\label{cnj:SquareFreeSpec2}
There exist infinitely many squarefree radicands $d$ of second species
such that $L=\mathbb{Q}(\sqrt[3]{d})$ is an $\mathrm{M}0$-field.
\end{conjecture}


\begin{example}
\label{exm:SquareFreeSpec2Var1}
We prove two defects in \cite[\S\ 6, Tbl. 2, p. 273]{HCW1982},
as claimed in Theorem \ref{thm:WilliamsErrors},
both of species $2$, $d\equiv\pm 1\,(\mathrm{mod}\,9)$.
They can be treated by the first variant of Theorem \ref{thm:SquareFreeSpec2}.
\begin{itemize}
\item
Let $d=1\,430=2\cdot 5\cdot 11\cdot 13$ and $n=1\,100=2^2\cdot 5^2\cdot 11$.
Then $d_1=11$, $d_2=2\cdot 5=10$, $d_3=13$,
and \eqref{eqn:M0FieldSpec2Coarse} is satisfied with
$d_1d_3=11\cdot 13=143>100=10^2=d_2^2$, $d_3^2=13^2=169>110=11\cdot 10=d_1d_2$, $d_2d_3=10\cdot 13=130>121=11^2=d_1^2$.
Furthermore, \eqref{eqn:M0FieldSpec2Var1} is satisfied with
$-d_1=-11\equiv d_2=10\equiv d_3=13\,(\mathrm{mod}\,3)$, $d_3=13<22=2\cdot 11=2d_1$, $d_3=13<24.49\approx 2.449\cdot 10\approx \sqrt{6}d_2$.
Therefore, $L=\mathbb{Q}(\sqrt[3]{1\,430})$ is an $\mathrm{M}0$-field.
\item
Let $d=12\,673=19\cdot 23\cdot 29$ and $n=10\,051=19\cdot 23^2$.
Then $d_1=19$, $d_2=23$, $d_3=29$,
and \eqref{eqn:M0FieldSpec2Coarse} is satisfied with
$d_1d_3=19\cdot 29=551>529=23^2=d_2^2$, $d_3^2=29^2=841>437=19\cdot 23=d_1d_2$, $d_2d_3=23\cdot 29=667>361=19^2=d_1^2$.
Furthermore, \eqref{eqn:M0FieldSpec2Var1} is satisfied with
$-d_1=-19\equiv d_2=23\equiv d_3=29\,(\mathrm{mod}\,3)$, $d_3=29<38=2\cdot 19=2d_1$, $d_3=29<56.34\approx 2.449\cdot 23\approx\sqrt{6}d_2$.
Consequently, $L=\mathbb{Q}(\sqrt[3]{12\,673})$ is an $\mathrm{M}0$-field.
We point out that this radicand is of the fifth form in the Main Theorem \cite[Thm. 1.1]{AMI}.
\end{itemize}
\end{example}


\begin{example}
\label{exm:1430}
Up to now, no examples of $\mathrm{M}0$-fields of species $2$ were known.
Since $d=1\,430$ was the first discovered radicand of such an exotic field $L=\mathbb{Q}(\sqrt[3]{d})$,
we present some details of the actual execution of Voronoi's algorithm.
The procedure starts at the trivial unit $\theta_0=1$, respectively $\phi_0=1$, and constructs the chain of lattice minima,
$\Theta$ of the maximal order $\mathcal{O}_{L}$, respectively $\Phi$ of the non-maximal order $\mathcal{O}_{L,0}$,
in direction of decreasing height $h=z$ and increasing radius $r=\sqrt{x^2+y^2}$ in Minkowski signature space $\mathbb{R}^3$,
and stops at the inverse fundamental unit $0<\theta_{-\ell}=\varepsilon^{-1}<1$, respectively $0<\phi_{-\ell_0}=\varepsilon_0^{-1}<1$,
as illustrated in Figure \ref{fig:LatticeMinima}.
In this particular example the unit groups of maximal order and suborder coincide and $\varepsilon_0=\varepsilon$.
Before the period ended at length $\ell_0=48$ we found two principal factors at characteristic locations
$j=-16=\frac{1}{3}\cdot (-48)$ exactly and $j=-34\approx\frac{2}{3}\cdot (-48)$ approximately:
\begin{equation}
\label{eqn:Minima1430}
\begin{aligned}
\beta:=\phi_{-16} &= -28\,490-13\,120\delta+1\,389\bar{\delta}, \\
\alpha:=\phi_{-34} &= -5\,130\,804\,470+350\,650\,663\delta+9\,298\,918\bar{\delta}, \\
\varepsilon_0=\phi_{-48} &= -6\,074\,553\,925\,441-689\,057\,082\,849\delta+109\,019\,548\,011\bar{\delta}.
\end{aligned}
\end{equation}
For instance the norm of $\beta=x+y\delta+z\bar{\delta}$ can be computed with the homogeneous pure cubic norm form
$N(\beta)=x^3+d\cdot y^3+d^2\cdot z^3-3d\cdot xyz$
\begin{equation*}
\label{eqn:Norm1430}
\begin{aligned}
&= -28\,490^3+1\,430\cdot (-13\,120^3)+2\,044\,900\cdot 1\,389^3-3\cdot1\,430\cdot (-28\,490)\cdot (-13\,120)\cdot 1\,389 \\
&= -23\,124\,766\,049\,000-3\,229\,516\,759\,040\,000+5\,479\,977\,964\,418\,100-2\,227\,336\,439\,328\,000 \\
&= 1\,100=2^2\cdot 5^2\cdot 11.
\end{aligned}
\end{equation*}
In Table \ref{tbl:OrdersCompared},
we compare the crucial locations in the chains of both orders.
By the general theory of principal factors, we have the characteristic relations
$\varepsilon_0^{-1}=\frac{\beta^3}{N(\beta)}$ and $\varepsilon_0^{-2}=\frac{\alpha^3}{N(\alpha)}$,
which shows that Voronoi's algorithm can be terminated at $\beta$ already, only a third of the period, to get the fundamental unit.
Of course, by Example \ref{exm:SquareFreeSpec2Var1}, we cannot find principal factors in the chain $\Theta$.
However, instead we encounter the \textit{shadows} of $\beta$ and $\alpha$ in the maximal order,
that is, the actual lattice minima within the norm cylinders of $\beta$ and $\alpha$:
\begin{equation}
\label{eqn:Shadows1430}
\begin{aligned}
\theta_{-17} &= \frac{1}{3}(56\,557 +28\,328\delta-2\,960\bar{\delta})
= \frac{1}{3}(-1+\frac{1}{10}\delta+\frac{1}{110}\bar{\delta})\cdot\phi_{-16}, \\
\theta_{-28} &= \frac{1}{3}(-112\,505\,639+13\,815\,812\delta+339\,929\bar{\delta}), \\
\theta_{-35} &= \frac{1}{3}(-8\,480\,403\,749-236\,672\,041\delta+87\,819\,928\bar{\delta})
= \frac{1}{3}(1+\frac{1}{11}\delta-\frac{1}{110}\bar{\delta})\cdot\phi_{-34}.
\end{aligned}
\end{equation}

\noindent
The shadow norms $N(\theta_{-17})=239$ and $N(\theta_{-35})=183$
can be computed with the results in \cite[\S\ 4, pp. 268--271]{HCW1982}. 
As opposed to the principal factor norms, the shadow norms are not unique,
and this fact causes complications, since for instance $\theta_{-28}$ with norm $183$
has nothing to do with principal factors, indicated by the symbol $\lightning$.

\renewcommand{\arraystretch}{1.0}

\begin{table}[ht]
\caption{First primitive periods of both orders compared}
\label{tbl:OrdersCompared}
\begin{center}
\begin{tabular}{|rrrc||rrr|}
\hline
 \multicolumn{4}{|c||}{maximal order $\mathcal{O}_{L}$}   & \multicolumn{3}{|c|}{non-maximal order $\mathcal{O}_{L,0}$}    \\
\hline
   $i$ &         $\theta_i$ & norm of $\theta_i$ &              &   $j$ &                                  $\phi_j$ & norm of $\phi_j$ \\
\hline
   $0$ &                $1$ &                $1$ &              &   $0$ &                                       $1$ &              $1$ \\
 $-17$ &     $\theta_{-17}$ &              $239$ & $\checkmark$ & $-16$ &  $\beta=\sqrt[3]{1100\varepsilon_0^{-1}}$ &           $1100$ \\
 $-28$ &     $\theta_{-28}$ &              $183$ & $\lightning$ &       &                                           &                  \\
 $-35$ &     $\theta_{-35}$ &              $183$ & $\checkmark$ & $-34$ & $\alpha=\sqrt[3]{1210\varepsilon_0^{-2}}$ &           $1210$ \\
 $-50$ & $\varepsilon^{-1}$ &                $1$ &              & $-48$ &                      $\varepsilon_0^{-1}$ &              $1$ \\
\hline
\end{tabular}
\end{center}
\end{table}

\noindent
Hence $L=\mathbb{Q}(\sqrt[3]{1\,430})$ is the first $\mathrm{M}0$-field of species $2$ and type $\beta$.
It has inadvertently been overlooked for some reason by H. C. Williams in \cite[Tbl. 2, p. 273]{HCW1982}.
\end{example}


\begin{example}
\label{exm:SquareFreeSpec2Var2}
Outside of the range $d<15\,000$ of radicands in the computations of \cite[\S\ 6, Tbl. 2, p. 273]{HCW1982}
there also occur examples of the second variant of Theorem \ref{thm:SquareFreeSpec2}.
\begin{itemize}
\item
Let $d=33\,337=17\cdot 37\cdot 53$ and $n=15\,317=17^2\cdot 53$.
Then $d_1=53$, $d_2=17$, $d_3=37$,
and \eqref{eqn:M0FieldSpec2Coarse} is satisfied with
$d_1d_3=53\cdot 37=1961>289=17^2=d_2^2$, $d_3^2=37^2=1369>901=53\cdot 17=d_1d_2$, $d_2d_3=17\cdot 37=629<2809=53^2=d_1^2$.
Unfortunately, \eqref{eqn:M0FieldSpec2Var2} with
$d_1=53\equiv d_2=17\equiv -d_3=-37\,(\mathrm{mod}\,3)$ is not satisfied,
since both inequalities $d_1=53>41.64\approx 2.449\cdot 17\approx\sqrt{6}d_2$ and $d_3=37>34=2\cdot 17=2d_2$
are in the false direction so that the fine criteria of Theorem \ref{thm:PrincipalFactorMinima} must be applied.
However, the field is interesting for another reason, since all prime factors are $\equiv\pm 1\,(\mathrm{mod}\,9)$
and thus the multiplicity of the conductor $f=d$ is given by $m(f)=2^3\cdot X_{-1}=2^3\cdot\frac{1}{2}=4$
giving rise to one of the rare \textit{quartets} of second species.
\item
Let $d=52\,417=23\cdot 43\cdot 53$ and $n=22\,747=23^2\cdot 43$.
Then $d_1=43$, $d_2=23$, $d_3=53$,
and \eqref{eqn:M0FieldSpec2Coarse} is satisfied with
$d_1d_3=43\cdot 53=2279>529=23^2=d_2^2$, $d_3^2=53^2=2809>989=43\cdot 23=d_1d_2$, $d_2d_3=23\cdot 53=1219<1849=43^2=d_1^2$.
Furthermore, \eqref{eqn:M0FieldSpec2Var2} is satisfied with
$-d_1=-43\equiv d_2=23\equiv d_3=53\,(\mathrm{mod}\,3)$, $d_1=43<46=2\cdot 23=2d_2$, $d_3=53<56.34\approx 2.449\cdot 23\approx\sqrt{6}d_2$.
Consequently, $L=\mathbb{Q}(\sqrt[3]{52\,417})$ is an $\mathrm{M}0$-field.
We point out that this radicand is of the seventh form in  \cite[Thm. 1.1]{AMI}.
\end{itemize}
\end{example}


\begin{theorem}
\label{thm:SquarePartSpec1B}
Let the square-part radicand $d=d_3d_4^2$ be of species $1\mathrm{a}$, $d\equiv\pm 2,\pm 4\,(\mathrm{mod}\,9)$,
and assume there exists a principal factor $\alpha\in\mathcal{O}_{L}$
with norm $n=9d_4$, minimal in the first coset $\lbrace 9d_4,9d_3,9d_3^2d_4^2\rbrace$, that is
\begin{equation}
\label{eqn:FieldSpec1BCoarse}
d_4<d_3,
\end{equation}
then $\bar{n}=3d_4^2$ is minimal in the second coset $\lbrace 3d_4^2,3d_3d_4,3d_3^2\rbrace$ and $\bar{\alpha}=\alpha^2/3$.
Denote by $Z_+$ the unique positive zero of the univariate polynomial $Q_4(X):=X^4+X^3+X-8\in\mathbb{Z}\lbrack X\rbrack$,
that is, $Z_+\approx 1.40080587094953$ with cube $Z_+^3\approx 2.74874124930414$.
Further, put $C_1:=(-1+\sqrt{33})/2\approx 2.37228132326901$ with cube $C_1^3\approx 13.3505319094211$.
Then,
\begin{itemize}
\item
$L$ is an $\mathrm{M}0$-field (neither $\alpha\in\mathrm{Min}(\mathcal{O}_{L})$ nor $\bar{\alpha}\in\mathrm{Min}(\mathcal{O}_{L})$)
$\Longleftrightarrow$
\begin{equation}
\label{eqn:M0FieldSpec1B}
d_3\equiv -d_4\,(\mathrm{mod}\,3), \quad d_3< Z_+^3\cdot d_4.
\end{equation}
\item
$L$ is an $\mathrm{M}1$-field ($\alpha\not\in\mathrm{Min}(\mathcal{O}_{L})$ but $\bar{\alpha}\in\mathrm{Min}(\mathcal{O}_{L})$)
$\Longleftrightarrow$
\begin{equation}
\label{eqn:M1FieldSpec1B}
\begin{aligned}
\text{either } & d_3\equiv -d_4\,(\mathrm{mod}\,3), \quad Z_+^3\cdot d_4\le d_3<8\cdot d_4 \\
\text{or }     & d_3\equiv d_4\,(\mathrm{mod}\,3), \quad d_3<C_1^3\cdot d_4. \\
\end{aligned}
\end{equation}
\item
$L$ is an $\mathrm{M}2$-field (both, $\alpha\in\mathrm{Min}(\mathcal{O}_{L})$ and $\bar{\alpha}\in\mathrm{Min}(\mathcal{O}_{L})$)
$\Longleftrightarrow$
\begin{equation}
\label{eqn:M2FieldSpec1B}
\begin{aligned}
\text{either } & d_3\equiv -d_4\,(\mathrm{mod}\,3), \quad 8\cdot d_4\le d_3 \\
\text{or }     & d_3\equiv d_4\,(\mathrm{mod}\,3), \quad C_1^3\cdot d_4\le d_3. \\
\end{aligned}
\end{equation}
\end{itemize}
\end{theorem}


\begin{proof}
We begin by seeking conditions for $\alpha\in\mathrm{Min}(\mathcal{O}_{L})$.
The normalized radicals are $1<\gamma=\delta/d_4$, $1<\bar{\gamma}=\bar{\delta}$.
Their cubes are $1<\gamma^3=\frac{d_3d_4^2}{d_4^3}=\frac{d_3}{d_4}<d_3^2d_4=\bar{\gamma}^3$,
whence $\min(\gamma,\bar{\gamma})=\gamma$.
Their product is $y=\gamma\bar{\gamma}=d_3$.
The congruence invariants are $u_1\equiv d_3d_4\,(\mathrm{mod}\,3)$ and $u_2\equiv d_4\,(\mathrm{mod}\,3)$.
Thus, we have four cases according to Formula \eqref{eqn:Escalatory9} in Theorem \ref{thm:PrincipalFactorMinima}:

If $d_3\equiv d_4\equiv 1\,(\mathrm{mod}\,3)$, then $(u_1,u_2)=(1,1)$ and \\
$\alpha\in\mathrm{Min}(\mathcal{O}_{L})$ $\Longleftrightarrow$ $\gamma<C_1$ $\Longleftrightarrow$ $d_3<C_1^3\cdot d_4$.

If $d_3\equiv d_4\equiv -1\,(\mathrm{mod}\,3)$, then $(u_1,u_2)=(1,-1)$ and \\
$\alpha\in\mathrm{Min}(\mathcal{O}_{L})$ $\Longleftrightarrow$ $\gamma<C_1$ $\Longleftrightarrow$ $d_3<C_1^3\cdot d_4$,
since $\bar{\gamma}<2$ $\Longrightarrow$ $\gamma<\bar{\gamma}<2<C_1$.

If $-d_3\equiv d_4\equiv 1\,(\mathrm{mod}\,3)$, then $(u_1,u_2)=(-1,1)$ and \\
$\alpha\in\mathrm{Min}(\mathcal{O}_{L})$ $\Longleftrightarrow$ $\gamma<2$ $\Longleftrightarrow$ $d_3<2^3\cdot d_4$.
(Note that the smallest possible square-part radicand is $12=2^2\cdot 3$,
whence $\bar{\gamma}=\bar{\delta}=d_3^2d_4\ge 12>C_1>2$.)

If $d_3\equiv -d_4\equiv 1\,(\mathrm{mod}\,3)$, then $(u_1,u_2)=(-1,-1)$ and \\
$\alpha\in\mathrm{Min}(\mathcal{O}_{L})$ $\Longleftrightarrow$ $\gamma<2$ $\Longleftrightarrow$ $d_3<8\cdot d_4$.
Herewith, the first coset is done.

\noindent
We turn to the second coset. The basic assumption $d_4<d_3$ in Formula \eqref{eqn:FieldSpec1BCoarse}
is equivalent with minimality of $n=9d_4$ in the first coset and minimality of $\bar{n}=3d_4^2$ in the second coset.
However, $\alpha^2$ has norm $81d_4^2$ and thus $\bar{\alpha}=\alpha^2/3$ has norm $\bar{n}$.
The new non-trivial canonical divisors of $\bar{n}=3d_4^2=3c_5^2$ are $c_3=d_3$ (fixed) and $c_5=d_4$ (twisted).
Therefore, the new congruence invariants are $u_1\equiv c_3c_5=d_3d_4\,(\mathrm{mod}\,3)$ as before,
but $u_2\equiv 1\,(\mathrm{mod}\,3)$ is constant.
Consequently, we have only two cases according to Formula \eqref{eqn:Escalatory3} in Theorem \ref{thm:PrincipalFactorMinima},
since $(u_1,u_2)=(1,-1)$ and $(u_1,u_2)=(-1,-1)$ cannot occur:

If $d_3\equiv d_4\,(\mathrm{mod}\,3)$, then $(u_1,u_2)=(1,1)$ and
$\alpha\in\mathrm{Min}(\mathcal{O}_{L})$.

If $d_3\equiv -d_4\,(\mathrm{mod}\,3)$, then $(u_1,u_2)=(-1,1)$ and \\
$\alpha\in\mathrm{Min}(\mathcal{O}_{L})$ $\Longleftrightarrow$ $P_4(-\gamma,y)<0$ $\Longleftrightarrow$ $P_4(\bar{\gamma},y)<0$.

Now we come to a phenomenon which is very peculiar for the present situation.
The new normalized radicals are
$\gamma=\delta/c_5=\delta/d_4=\sqrt[3]{\frac{d_3d_4^2}{d_4^3}}=\sqrt[3]{\frac{d_3}{d_4}}$ as before,
but $\bar{\gamma}=\bar{\delta}/c_5=\bar{\delta}/d_4=\sqrt[3]{\frac{d_3^2d_4}{d_4^3}}=\sqrt[3]{\frac{d_3^2}{d_4^2}}=\gamma^2$,
and their product is $y=\gamma\bar{\gamma}=\frac{d_3}{d_4}=\gamma^3$.
Actual substitution into $P_4(X,Y)=X^4-X^3+X^2Y-8X^2+XY+Y^2$ yields
$P_4(-\gamma,y)=P_4(-\gamma,\gamma^3)=\gamma^4+\gamma^3+\gamma^2\gamma^3-8\gamma^2-\gamma\gamma^3+\gamma^6=\gamma^2(\gamma^4+\gamma^3+\gamma-8)$
and similarly $P_4(\bar{\gamma},y)=P_4(\gamma^2,\gamma^3)=\gamma^4(\gamma^4+\gamma^3+\gamma-8)$.

Since $\gamma\ge 1$, we obtain $P_4(-\gamma,y)<0$ $\Longleftrightarrow$ $Q_4(\gamma)=\gamma^4+\gamma^3+\gamma-8<0$
$\Longleftrightarrow$ $\gamma<Z_{+}$ $\Longleftrightarrow$ $y=\frac{d_3}{d_4}=\gamma^3<Z_{+}^3$ $\Longleftrightarrow$ $d_3<Z_{+}^3\cdot d_4$,
because the negative zero $Z_{-}$ of $Q_4(X)$ is irrelevant.
\end{proof}


\begin{example}
\label{exm:SquarePartSpec1B}
We confirm six results in \cite[\S\ 6, Tbl. 2, p. 273]{HCW1982},
as reproduced in Theorem \ref{thm:WilliamsErrors},
all of species $1\mathrm{b}$, $d\equiv\pm 2,\pm 4\,(\mathrm{mod}\,9)$.
They can be treated by Theorem \ref{thm:SquarePartSpec1B}.
\begin{itemize}
\item
Let $d=833=7^2\cdot 17$ and $n=63=3^2\cdot 7$.
Then $d_3=17$, $d_4=7$,
and \eqref{eqn:FieldSpec1BCoarse} is satisfied with
$d_4=7<17=d_3$.
Further, \eqref{eqn:M0FieldSpec1B} is satisfied with
$d_3=17\equiv -d_4=-7\equiv -1\,(\mathrm{mod}\,3)$, $d_3=17<19.24\approx 2.7487\cdot 7\approx Z_{+}^3\cdot d_4$.
Therefore, $L=\mathbb{Q}(\sqrt[3]{833})$ is an $\mathrm{M}0$-field.
\item
Let $d=1\,573=11^2\cdot 13$ and $n=99=3^2\cdot 11$.
Then $d_3=13$, $d_4=11$,
and \eqref{eqn:FieldSpec1BCoarse} is satisfied with
$d_4=11<13=d_3$.
Also, \eqref{eqn:M0FieldSpec1B} is satisfied with
$d_3=13\equiv -d_4=-11\equiv 1\,(\mathrm{mod}\,3)$, $d_3=13<30.2\approx 2.7487\cdot 11\approx Z_{+}^3\cdot d_4$,
and $L=\mathbb{Q}(\sqrt[3]{1\,573})$ is an $\mathrm{M}0$-field.
\item
Let $d=4\,901=13^2\cdot 29$ and $n=117=3^2\cdot 13$.
Then $d_3=29$, $d_4=13$,
and \eqref{eqn:FieldSpec1BCoarse} is satisfied with
$d_4=13<29=d_3$.
Also, \eqref{eqn:M0FieldSpec1B} is satisfied with
$d_3=29\equiv -d_4=-13\equiv -1\,(\mathrm{mod}\,3)$, $d_3=29<35.73\approx 2.7487\cdot 13\approx Z_{+}^3\cdot d_4$,
and $L=\mathbb{Q}(\sqrt[3]{4\,901})$ is an $\mathrm{M}0$-field.
\item
Let $d=6\,358=2\cdot 11\cdot 17^2$ and $n=153=3^2\cdot 17$.
Then $d_3=22$, $d_4=17$,
and \eqref{eqn:FieldSpec1BCoarse} is satisfied with
$d_4=17<22=d_3$.
Also, \eqref{eqn:M0FieldSpec1B} is satisfied with
$d_3=22\equiv -d_4=-17\equiv 1\,(\mathrm{mod}\,3)$, $d_3=22<46.7\approx 2.7487\cdot 17\approx Z_{+}^3\cdot d_4$,
and $L=\mathbb{Q}(\sqrt[3]{6\,358})$ is an $\mathrm{M}0$-field.
\item
Let $d=8\,959=17^2\cdot 31$ and $n=153=3^2\cdot 17$.
Then $d_3=31$, $d_4=17$,
and \eqref{eqn:FieldSpec1BCoarse} is satisfied with
$d_4=17<31=d_3$.
Also, \eqref{eqn:M0FieldSpec1B} is satisfied with
$d_3=31\equiv -d_4=-17\equiv 1\,(\mathrm{mod}\,3)$, $d_3=31<46.7\approx 2.7487\cdot 17\approx Z_{+}^3\cdot d_4$.
Therefore, $L=\mathbb{Q}(\sqrt[3]{8\,959})$ is an $\mathrm{M}0$-field.
\item
Let $d=14\,801=19^2\cdot 41$ and $n=171=3^2\cdot 19$.
Then $d_3=41$, $d_4=19$,
and \eqref{eqn:FieldSpec1BCoarse} is satisfied with
$d_4=19<41=d_3$.
Also, \eqref{eqn:M0FieldSpec1B} is satisfied with
$d_3=41\equiv -d_4=-19\equiv -1\,(\mathrm{mod}\,3)$, $d_3=41<52.2\approx 2.7487\cdot 19\approx Z_{+}^3\cdot d_4$,
and $L=\mathbb{Q}(\sqrt[3]{14\,801})$ is an $\mathrm{M}0$-field.
\end{itemize}
Note that all these radicands, except $6\,358$, are of the third form in \cite[Thm. 1.1]{AMI}.
\end{example}


In Table \ref{tbl:JustificationSpec2},
we show for some radicands $d$ of $\mathrm{M}0$-fields
whether the proof is possible either by coarse rational integer criteria $y=\gamma\bar{\gamma}<C$ ($\checkmark$)
or only by fine multiprecision criteria $P_2(u_1\gamma,u_2\bar{\gamma})<B$ involving irrationalities, when $y\ge C$ ($\lightning$).

\renewcommand{\arraystretch}{1.0}

\begin{table}[ht]
\caption{Justifications for $\mathrm{M}0$-fields of species $2$ with coarse and fine criteria}
\label{tbl:JustificationSpec2}
\begin{center}
\begin{tabular}{|r||rrc|rr||rrc|rr|}
\hline
       $d$          & \multicolumn{5}{|c||}{first coset of $\alpha$}           & \multicolumn{5}{|c|}{second coset of $\beta$}            \\
\hline
                    &      $y$ &      $C$ &              &    $P_2$ &      $B$ &      $y$ &      $C$ &              & $P_2$ &         $B$ \\
\hline
           $1\,430$ & $1.3000$ & $2.4494$ & $\checkmark$ & $4.5812$ & $9.0000$ & $1.1818$ & $2.0000$ & $\checkmark$ & $4.6919$ & $9.0000$ \\
          $12\,673$ & $1.2608$ & $2.4494$ & $\checkmark$ & $4.5713$ & $9.0000$ & $1.5263$ & $2.0000$ & $\checkmark$ & $5.5960$ & $9.0000$ \\
          $20\,539$ & $2.0434$ & $2.4494$ & $\checkmark$ & $6.2265$ & $9.0000$ & $2.4736$ & $2.0000$ & $\lightning$ & $8.7714$ & $9.0000$ \\
          $33\,337$ & $2.1764$ & $2.0000$ & $\lightning$ & $8.8258$ & $9.0000$ & $3.1176$ & $2.4494$ & $\lightning$ & $7.7183$ & $9.0000$ \\
          $52\,417$ & $2.3043$ & $2.4494$ & $\checkmark$ & $6.3921$ & $9.0000$ & $1.8695$ & $2.0000$ & $\checkmark$ & $7.3155$ & $9.0000$ \\
\hline
\end{tabular}
\end{center}
\end{table}




\section{Conclusion}
\label{s:Conclusion}

\noindent
In our previous work \cite{AMI}, have characterized in all Kummer extensions $\mathit{k}/\mathit{k}_0$,
which possess a relative $3$-genus field $\mathit{k}^{\ast}$ with elementary bicyclic Galois group $\operatorname{Gal}(\mathit{k}^{\ast}/\mathit{k})$. The underlying pure cubic subfields $L=\mathbb{Q}(\sqrt[3]{d})$ partially reveal
the rare behavior that none of the generators of primitive ambiguous principal ideals
occurs among the lattice minima of the maximal order $\mathcal{O}_{L}$.
We have given necessary and sufficient conditions for these exotic fields.
Since their existence has an unpleasant impact on the classification of pure cubic fields $L$
by means of Voronoi's algorithm, we have developed and implemented a marvellous technique
for unambiguously determining the principal factorization type of $L$,
thereby correcting serious defects in earlier tables.


\section{Acknowledgements}
\label{s:Thanks}

\noindent
This paper is respectfully dedicated to Professor H. C. Williams
on the occasion of his 75th birthday.

The fourth author gratefully acknowledges that his research was supported by the
Austrian Science Fund (FWF): P 26008-N25.


\begin{quote}

Abdelmalek AZIZI, Moulay Chrif ISMAILI and Siham AOUISSI  \\
Department of Mathematics and Computer Sciences, \\
Mohammed first University, \\
60000 Oujda - Morocco, \\
abdelmalekazizi@yahoo.fr, mcismaili@yahoo.fr, aouissi.siham@gmail.com. \\

Daniel C. MAYER \\
Naglergasse 53, 8010 Graz, Austria, \\
algebraic.number.theory@algebra.at \\
URL: http://www.algebra.at. \\

Mohamed TALBI \\
Regional Center of Professions of Education and Training, \\
60000 Oujda - Morocco, \\
ksirat1971@gmail.com.
\end{quote}

\end{document}